\documentclass[12pt, reqno]{amsart}

\usepackage{version}
\usepackage[T1]{fontenc} 
\usepackage[utf8]{inputenc} 
\usepackage{amsmath}
\usepackage{amssymb}
\usepackage{mathrsfs}
\usepackage{amsthm}
\usepackage{latexsym}
\usepackage{wasysym}
\usepackage{dcpic, pictex}
\usepackage{graphicx}
\usepackage{tikz-cd}

\newtheorem{theorem}{Theorem}[section]
\newtheorem{lemma}[theorem]{Lemma}
\newtheorem{corollary}[theorem]{Corollary}
\newtheorem{proposition}[theorem]{Proposition}

\theoremstyle{definition}
\newtheorem{definition}[theorem]{Definition}
\newtheorem{example}[theorem]{Example}

\newtheorem{remark}[theorem]{Remark}

\author{Matteo Fiacchi}

\def\A{{\mathcal{A}}}

\def\D{{\mathbb{D}}}
\def\C{{\mathbb{C}}}
\def\R{{\mathbb{R}}}

\def\N{{\mathbb{N}}}

\def\L{{\mathscr{L}}}

\newcommand{\id}{{\sf Id}}
\newcommand{\re}{{\sf Re}}

\begin{document}
\title{Gromov hyperbolicity of pseudoconvex finite type domains in $\C^2$} 

\thanks{ Partially supported by {\sl PRIN Real and Complex Manifolds: Topology,
	Geometry and holomorphic dynamics} n. 2017JZ2SW5 and by MIUR
	Excellence Department Project awarded to the Department of
	Mathematics, University of Rome Tor Vergata, CUP E83C18000100006.}

\address{M. Fiacchi: Dipartimento di Matematica\\
Universit\`{a} di Roma \textquotedblleft Tor Vergata\textquotedblright\ \\Via Della Ricerca Scientifica 1, 00133 \\
Roma, Italy} \email{fiacchi@mat.uniroma2.it}
\begin{abstract}
We prove that  every bounded smooth domain of finite D'Angelo type in $\C^2$ endowed with the Kobayashi distance is Gromov hyperbolic and its  Gromov boundary is canonically homeomorphic to the Euclidean boundary.  We also show that any domain in $\C^2$ endowed with the Kobayashi distance is Gromov hyperbolic provided there exists a sequence of automorphisms that converges to a smooth boundary point of finite D'Angelo type.
\end{abstract}
\maketitle

\section{Introduction}

The Kobayashi metric is an important invariant metric that can be used to study domains in $\C^n$ and the holomorphic functions defined on them. In this paper we study the relation between the $CR$ structure of the boundary and the Gromov hyperbolicity of a domain endowed with the Kobayashi distance. Gromov hyperbolic spaces are a generalization of negative curved spaces in the context of metric spaces, and this concept fits very well to our problem since the Kobayashi metric is in general not Riemannian. 

The problem has been much studied in recent years, for instance Balogh and Bonk \cite{BB} proved that bounded strongly pseudoconvex domains endowed with the Kobayashi distance are Gromov hyperbolic. In different works, Gaussier-Seshadri (in the bounded smooth case \cite{GS}), Nikolov-Thomas-Trybula ($\mathcal{C}^{1,1}$ boundary in $\C^2$ \cite{NTT}) and Zimmer (in the general case \cite{Zim1}) proved that,  in  convex domains the existence of  holomorphic disks in the boundary prevents the domains from being Gromov hyperbolic with respect to the Kobayashi distance. Moreover, in the same paper, Zimmer gave the following complete characterization: let $D\subset\C^d$ be a smooth {\sl convex} domain. Then, $D$ endowed with the Kobayashi distance is Gromov hyperbolic if and only if  $D$ is of D'Angelo finite type. 

If $M$ is a smooth $CR$ hypersurface in $\C^d$ that can be written as $M=\{\rho=0\}$ where $\rho$ is a $\mathcal{C}^\infty$ function, a point $p\in M$ is of \textit{D'Angelo type} $m$ if 
$$m=\sup\left\{\frac{\nu(\rho\circ f)}{\nu(f)}:f:\D\longrightarrow\C^d\ \ \mbox{holomorphic,}\ \ f(0)=p\right\}$$
where $\nu(\cdot)$ indicates the order of vanishing, and $M$ is of \textit{finite type} is every point has finite D'Angelo type. A point in a $CR$ hypersurface is strongly pseudoconvex if and only if its D'Angelo type is 2, so smooth pseudoconvex domains with finite type boundary (\textit{finite type domains}, for short) can be seen as a generalization of strongly pseudoconvex domains. 

Our goal is to study the Gromov hyperbolicity (with respect to the Kobayashi distance) of finite type domains, with no assumption of convexity. Exiting the category of convex domains, the problem of deciding the Gromov hyperbolicity of a domain becomes very subtle. For instance,  Zimmer \cite{Zim2} found an example of a Gromov hyperbolic, non-convex and non-smooth, domain  with an analytic disk in the boundary. The problem of determining whether a domain endowed with the Kobayashi distance is Gromov hyperbolic is pretty much related to the existence of suitable estimates for the Kobayashi distance near the boundary. Such estimates are known for  convex and strongly pseudoconvex domains, but there are very few outside these categories. In the present paper, we exploit Catlin estimates \cite{Cat}, which hold in dimension 2. Our main result is the following (here $d^K_\Omega$ denotes the Kobayashi distance of a domain $\Omega\subset\C^d$):

\begin{theorem}\label{mainth}
Let $\Omega\subset\C^2$ be a bounded smooth pseudoconvex finite type domain. Then $(\Omega, d^K_\Omega)$ is Gromov hyperbolic. Moreover, the identity map $\Omega\longrightarrow\Omega$
extends to a homeomorphism $\overline{\Omega}^G\longrightarrow\overline{\Omega}$, where $\overline{\Omega}^G$ denotes the Gromov compactification of $\Omega$ endowed with the Gromov topology and $\overline{\Omega}$ is the Euclidean closure of $\Omega$.
\end{theorem}

This theorem can be seen as a partial generalization (in $\C^2$)  of Zimmer's analogous result for convex finite type domains.

Following the ideas of Benoist in the context of Hilbert geometry \cite{Benoist}, later adapted by Zimmer in complex geometry, the strategy to prove the Gromov hyperbolicity is to use a scaling process.
In this paper, we use the Pinchuk scaling method: given a sequence of points $\{p_n\}_{n\in\N}$ in the domain $\Omega$ that converges to a finite type point in the boundary, we can find a family of automorphisms of $\C^2$ that maps $\Omega$ to a family of domains $\Omega_n$ that converges (in some sense) to a model domain $\Omega_\infty$ and sends $\{p_n\}_{n\in\N}$ to a convergent sequence in $\Omega_\infty$. The limit domain $\Omega_\infty$ is of the form
$$\{(z,w)\in\C^2: \re[w]+P(z)<0 \} $$
where $P:\C\longrightarrow\R$ is a subharmonic polynomial without harmonic terms and $P(0)=0$.
Then, the main issue is to show  that the Kobayashi distance is stable under this process. This is true in the convex case (see again \cite{Zim1}) but still unknown in general. In our setting, we are able to show that this is the case exploiting Catlin's estimates.

On the way to prove Theorem \ref{mainth}, we obtain  the following result, interesting by its own:

\begin{theorem}\label{ncautoth}
Let $\Omega\subset\C^2$ be a pseudoconvex domain. Suppose that there exists $\xi\in\partial\Omega$ such that $\partial\Omega$ is smooth and of finite type in a neighborhood of $\xi$ and there exists a sequence of automorphisms $\{\phi_k\}_{k\in\N}$ and a point $p\in\Omega$ such that $\phi_k(p)\rightarrow\xi$. Then $\Omega$ is complete Kobayashi hyperbolic and $(\Omega,d^K_\Omega)$ is Gromov hyperbolic.
\end{theorem}

It should be noticed that in the previous result we are not assuming boundedness or global smoothness of the domain. We refer the reader to, e.g. \cite{BP}, \cite{Gau} and \cite{IK} and bibliography therein for more about domains with a non-compact automorphisms group.  

As applications of our results, we study extension of biholomorphisms and commuting maps in finite type domains of $\C^2$. The canonical homeomorphism between the Gromov boundary (or any other ``intrinsic'' boundary) and the natural boundary of a domain, allows to prove that every biholomorphism---and more general, every quasi-isometry---extends as a homeomorphism on the boundary (see \cite{BB, BG1, BG2, BGZ}). Using our main result, we can add bounded smooth finite type domains to this list (see Theorem \ref{Thm:extension}). Finally, our main result  allows to obtain an extension of a result of Behan \cite{Behan}   for commuting holomorphic self-maps of the unit disc to the setting of finite type domains in $\C^2$ (see Theorem~\ref{Thm:Behan}).
\medbreak

\textbf{Acknowledgements:} The author would like to express his gratitude to Prof. Herv\'e Gaussier for his availability and for introducing him to Gromov hyperbolicity theory, and to Prof. Filippo Bracci for useful discussions and suggestions.

\section{Preliminaries}

\subsection{Gromov hyperbolic spaces} One of the possible reference for the Gromov hyperbolic spaces is Bridson and Haefliger’s book \cite{BH}.

We recall that a metric space $(X,d)$ is called \textit{proper} if the closed balls are compact.

\begin{definition} Let $(X,d)$ be a metric space, $I\subset\R$ be an interval and $A\geq1$ and $B\geq0$. A function $\sigma:I\longrightarrow X$ is
\begin{enumerate}
	\item a \textit{geodesic} if for each $s,t\in I$ 
	$$d(\sigma(s),\sigma(t)))=|t-s|. $$
	\item a $(A,B)$ \textit{quasi-geodesic} if for each $s,t\in I$
	$$A^{-1}|t-s|-B\leq d(\sigma(s),\sigma(t)))\leq A|t-s|+B.$$
\end{enumerate}
Finally, we use the word \textit{ray} if $I=[0,+\infty)$, and \textit{line} if $I=\R$.
\end{definition}	

A metric space is a \textit{geodesic space} if for any two points there is a geodesic that connects them.
A \textit{geodesic triangle} is a choice of three points in $X$ and geodesic
segments connecting these points.

\begin{definition} A proper geodesic metric space $(X,d)$ is called $\delta$\textit{-hyperbolic} if for every geodesic triangle and any point on any of the sides of the triangle is within distance $\delta$ of the other two sides. Finally, a metric space is \textit{Gromov hyperbolic} if it is $\delta$-hyperbolic for some $\delta\geq0$.
\end{definition}

For every $x,y,o\in X$ we define the \textit{Gromov product} as 
$$(x|y)_o:=\frac{1}{2}[d(x,o)+d(y,o)-d(x,y)].$$

The Gromov hyperbolic spaces have a natural definition of boundary: let $\A$ be the set of geodesic rays starting from a fixed point $o\in X$ and $\sim$ an equivalent relation on $\A$ given by
$$\sigma_1\sim\sigma_2 \ \Leftrightarrow \ \sup_{t\geq0}d(\sigma_1(t),\sigma_2(t))<+\infty.$$
We define the \textit{Gromov boundary} as $\partial^GX:=\A/\sim$, and the \textit{Gromov compactification} as $\overline{X}^G:=X\cup\partial^GX$. 

We can also define a topology on $\overline{X}^G$. For simplicity we introduce the following notation: if $\sigma:[0,T]\longrightarrow X$ is a geodesic we set $\sigma(\infty):=\sigma(T)$, and if $\sigma:[0,+\infty)\longrightarrow X$ is a geodesic ray we set $\sigma(\infty)$ as the equivalence class in $\partial^GX$. Now $\overline{X}^G$ has a topology where $\xi_n\in\overline{X}^G\rightarrow \xi\in\overline{X}^G$ if for every choice of geodesic $\sigma_n$ with $\sigma_n(0)=o$ and $\sigma_n(\infty)=\xi_n$ we have that every subsequence of $\{\sigma_n\}$ has a subsequence which converges locally uniformly to
a geodesic $\sigma$ with $\sigma(\infty)=\xi$.
Finally, it easy to see that the subspace topology of $X$ is the same topology that arise from $(X,d)$.

The interest on this compactification comes from the nice theorem about the extension of mappings. We say that a continuous bijective function $f:(X,d_X)\longrightarrow(Y,d_Y)$ is a \textit{quasi-isometry} if there exists $A\geq1$ and $B\geq0$ such that for every $p,q\in X$
$$A^{-1}d_X(p,q)-B\leq d_Y(f(p),f(q))\leq Ad_X(p,q)+B.$$
Finally, we recall that Gromov hyperbolicity is invariant for quasi isometries.

\begin{theorem}\label{extth}
Let $(X,d_X)$ and $(Y,d_Y)$ be two Gromov hyperbolic spaces, and $f:X\longrightarrow Y$ be a quasi-isometry. Then $f$ extends to an unique homeomorphism $f^*:\overline{X}^G\longrightarrow\overline{Y}^G.$
\end{theorem}

\subsection{Finsler metrics and Kobayashi distance}

In this section we will review some of the main concepts about Finsler metrics and in particular about Kobayashi metric. An introduction can be found respectively in \cite{AbPatr} and in \cite{Kob}.

\begin{definition}
	Let $\Omega\subset\C^d$ be a domain. A \textit{Finsler metric} is an upper semicontinuous function 
	$F:\Omega\times\C^d\longrightarrow[0,+\infty)$
	with the following properties
	\begin{enumerate}
		\item $F(p,v)>0$ for all $p\in\Omega$, $v\in\C^d$ $v\neq0$;
		\item $F(p,\lambda v)=|\lambda|F(p,v)$ for all $p\in\Omega$, $v\in\C^d$ and $\lambda\in\C$.
	\end{enumerate}
\end{definition}

Given a Finsler metric, we can define a distance on $\Omega$
$$d^F_\Omega(p,q)=\inf\left\{\int_0^1F(\gamma(t),\gamma'(t))dt: \gamma:[0,1]\longrightarrow\Omega, \ \mathcal{C}^1\ \mbox{piecewise},\  \gamma(0)=p,\gamma(1)=q \right\}.$$

An important Finsler metric in complex analysis is the Kobayashi metric: we define the following function
$$K_\Omega(p,v):=\inf\{|\lambda|: f:\D\longrightarrow\Omega\ \mbox{holomorphic},\ f(0)=p, df_0(\lambda)=v\}. $$
If for each $p\in\Omega$ and $v\neq0$ $K_\Omega(p,v)>0$ we say that $\Omega$ is \textit{Kobayashi hyperbolic}, and in this case $K_\Omega$ is a Finsler metric on $\Omega$ that we call \textit{Kobayashi metric}. Finally, the associate distance $d^K_\Omega$ is called \textit{Kobayashi distance}.

We recall the following fundamental result
\begin{proposition}\label{Kobdecr}
	Let $\Omega_1\subset\C^d$ and $\Omega_2\subset\C^q$ be two domains, and $f:\Omega_1\longrightarrow\Omega_2$ be an holomorphic function. Then for each $p,q\in\Omega_1$ and $v\in\C^d$ we have
	$$K_{\Omega_2}(f(p),df_p(v))\leq K_{\Omega_1}(p,v)$$
	and
	$$d^K_{\Omega_1}(f(p),f(q))\leq d^K_{\Omega_2}(p,q),$$
	moreover, if $f$ is a biholomorphism we have the equality in both the equations. 
\end{proposition}

	We conclude this section with a result about the regularity of quasi-geodesics.
	\begin{proposition}
		Let $\Omega\subset\C^d$ be a bounded taut domain. Then, for each $A\geq1$ there exists $C_A>0$ such that any $(A,0)$ quasi-geodesic $\sigma:I\longrightarrow\Omega$ is $C_A$-Lipschitz with respect to the Euclidean distance (and so there exists $\sigma'(t)$ almost everywhere), and 
		$$K_\Omega(\sigma(t),\sigma'(t))\leq A $$
		for almost every $t\in I$.
	\end{proposition}
	\proof
	The proof is analogous to \cite[Proposition 4.6]{BZ}, so we omit it.
	\endproof

\section{Catlin metric for finite type domains}

In this section we recall a Finsler metric which can be defined in a domain in $\C^2$ near a boundary point of finite type. This metric was introduced by Catlin in \cite{Cat} to have an estimate of Kobayashi metric near a boundary point of finite type.
This estimate can be seen as a generalization in $\C^2$ of the metric defined in \cite{BB} that works in strongly pseudoconvex domains.

Let $\Omega\subseteq\C^2$ be a pseudoconvex domain, $r$ its defining function. Let $\xi\in\partial\Omega$ such that $r$ is smooth of finite type in a neighborhood of $\xi$.
Suppose that $\frac{\partial r}{\partial w}(\xi)\neq0$, then in a neighborhood of $\xi$ we can define the following vector fields
$$L_r:=\frac{\partial}{\partial z}-\Bigl(\frac{\partial r}{\partial w}\Bigr)^{-1}\frac{\partial r}{\partial z}\frac{\partial}{\partial w},\ \  \mbox{and}\ \  N:=\frac{\partial}{\partial w}.$$
For any $j,k>0$, set
$$\mathscr{L}^{j,k}_r(p):=\underbrace{L_r\dots L_r}_{j-1\ times} \underbrace{\bar{L}_r\dots \bar{L}_r}_{k-1\ times} \mathscr{L}_r(p)(L_r,\bar{L}_r).$$
Now if $p\in\partial\Omega$ is point of type $m$ there exist $j_0,k_0$ with $j_0+k_0=m$ (see \cite{Cat}) such that 
$$\mathscr{L}_r^{j,k}(p)=0\ \ \ j+k<m, \ \ \mbox{and} $$
$$\mathscr{L}_r^{j_0,k_0}(p)\neq0. $$
We define
$$C_l^r(p)=\max\{|\mathscr{L}_r^{j,k}(p)|: j+k=l\}.$$
Let $X=aL_r+bN$ be an holomorphic tangent vector at $p$ and set
\begin{equation}\label{catmetric}
M_r(p,X):=\frac{|b|}{|r(p)|}+|a|\sum_{l=2}^{m}\Bigl(\frac{C^r_l(p)}{|r(p)|}\Bigr)^{\frac{1}{l}}.
\end{equation}

We denote with $d_r^C$ the distance associated to the Finsler metric $M_r$.

The interest in this metric is given by the following

\begin{theorem}[Catlin]\label{CatlinTh}\cite{Cat}
Let $\Omega\subset\C^2$ be a bounded smooth pseudoconvex domain with defining function $r$. Let $\xi\in\partial\Omega$ of finite type, then there exists a neighborhood $U$ of $\xi$ and $A\geq1$ such that 
$$A^{-1}M_r(x,v)\leq K_\Omega(x,v)\leq AM_r(x,v) $$
for each $x\in\Omega\cap U$ and $v\in\C^2$.
\end{theorem}

We will see in the next section the importance of the model domains. A model domain $\Omega_P$ is a domain of the form
$$\Omega_P=\{(z,w)\in\C^2: \re[w]+P(z)<0\},$$
where $P:\C\longrightarrow\R$ is a subharmonic polynomial without harmonic terms (i.e. if we write $P(z)=\sum_{j+k\leq l}a_{j,k}z^j\bar{z}^k$ then $a_{0,k}=a_{j,0}=0$ for each $0\leq j,k\leq l$) and $P(0)=0$.

If we set $r_P(z,w):=\re[w]+P(z)$ the defining function of $\Omega_P$, since $\frac{\partial r_P}{\partial w}\equiv\frac{1}{2}$ we can define the metric (\ref{catmetric}) in all $\Omega_P$ and we obtain the following formula.

\begin{lemma}\label{Catmod}
The Catlin metric in $\Omega_P$ is 
$$M_{r_P}((z,w),(x,y))=\frac{|y+2xP'(z)|}{|r_P(z,w)|}+|x|\sum_{l=2}^{m}\Bigl(\frac{A^P_l(z)}{|r_P(z,w)|}\Bigr)^{\frac{1}{l}},$$
where $r_P(z,w)=\re[w]+P(z)$ and $A^P_l(z)=\max\left\{\left|\frac{\partial^{j+k}P}{\partial z^j\partial \bar{z}^k}(z)\right|:j,k>0,\ j+k=l\right\}$.
\end{lemma}

Now we prove that the estimate in Theorem \ref{CatlinTh} works also in homogeneous model domains.

\begin{proposition}\label{cathom}
Let $H:\C\longrightarrow\R$ be subharmonic homogeneous polynomial without harmonic terms  of degree $m$ and $\Omega_H$ be its model domain. Then there exists $A\geq1$ such that 
\begin{equation}\label{eqmetrics}
A^{-1}M_{r_H}(x,v)\leq K_{\Omega_H}(x,v)\leq AM_{r_H}(x,v)
\end{equation}
for each $x\in\Omega_H$ and $v\in\C^2$.
\end{proposition}
\proof
We prove first two preliminary results.

\textbf{Claim 1:} For each $\xi\in\partial\Omega_H$, there exist a neighborhood $U$ of $\xi$ and $A\geq1$ such that (\ref{eqmetrics}) holds for each $x\in\Omega_H\cap U$.
\proof
We cannot apply directly Theorem \ref{CatlinTh} because $\Omega_H$ is not bounded, for this reason we need a localization result. Fix $R>0$ such that $\xi\in B_R(0)$ and $\Psi_R(x)=x^m\exp(-1\slash(x-R))\chi_{[R,\infty)}$ and set
$$r(z,w):=\re[w]+P(z)+\Psi_R(|z|^2+|w|^2)$$
and $\Omega:=\{r<0\}$. Notice that $\Omega$ is bounded, smooth, pseudoconvex, $\Omega\cap B_R(0)=\Omega_H\cap B_R(0)$ and $M_r=M_{r_H}$ in $\Omega\cap B_R(0)$. Now we can apply Theorem \ref{CatlinTh} to $\xi\in\partial\Omega$, then there exist a neighborhood $U\subset\Omega\cap B_R(0)$ of $\xi$ and $B\geq1$ such that for all $x\in \Omega\cap U$ and $v\in\C^2$ we have
$$B^{-1}M_{r}(x,v)\leq K_{\Omega}(x,v)\leq BM_{r}(x,v).$$
Moreover, by Theorem 2.2 in \cite{Yu}, we  have 
$$\lim_{\Omega\cap U\ni x\rightarrow\partial\Omega}\frac{K_\Omega(x,v)}{K_{\Omega_H}(x,v)}=1$$
and the limit is uniform, indeed there exists $C\geq1$ such that for all $x\in\Omega\cap U$, $v\in\C^2$
$$C^{-1}K_{\Omega}(x,v)\leq K_{\Omega_H}(x,v)\leq CK_{\Omega}(x,v).$$
Finally, the two estimates imply the statement if we set $A:=BC$.
\endproof
For all $\lambda>0$ we define $\Phi_\lambda(z,w):=(\lambda z,\lambda^m w)$. Since $H$ is homogeneous of degree $m$, $\Phi_\lambda$ is an automorphism of $\Omega_H$ for all $\lambda>0$.

\textbf{Claim 2:} The function $\Phi_\lambda$ is an isometry for $M_{r_H}$, that is for all $p\in\Omega_H$ and $v\in\C^2$ we have
$$\Phi^*_\lambda M_{r_H}(p,v)=M_{r_H}(p,v).$$
\proof
Elementary computation.
\endproof
Now we are ready for the proof of the Proposition. For all $p\in\Omega_H$ and $v\in\C^2$, let $\lambda>0$ small enough such that $\Phi_\lambda(p)\in U$ where $U$ is the neighborhood of Claim 1. Using Claim 2 and recalling that the automorphisms are isometries of the Kobayashi metric, we have
$$K_{\Omega_H}(p,v)=K_{\Omega_H}(\Phi_\lambda(p),(d\Phi_\lambda)_pv)\leq AM_{r_H}(\Phi_\lambda(p),(d\Phi_\lambda)_pv)=AM_{r_H}(p,v)$$
and similarly for the other inequality. 
\endproof

\section{Scaling of coordinates}

In this section we do a scaling process in order to have the following result (the technique is essentially the one in \cite{Bert}, but we repeat the proof because we need to study the behavior of Catlin's metric).

\begin{theorem}\label{scalingth}
Let $\Omega\subset\C^2$ be a pseudoconvex domain, and $\{\eta_n\}_{n\in\N}\subseteq\Omega$ a sequence that converges to a point $\eta_\infty\in\partial\Omega$ such that $\partial\Omega$ is smooth and of finite type in a neighborhood $U$ of $\eta_\infty$. Denote $r$ its defining function, and $M_r$ the Catlin metric defined in a neighborhood $U$ of $\eta_\infty$. Then, passing to a subsequence, there exists a sequence of automorphisms of $\C^2$ $\{\psi_n\}_{n\in\N}$ and a polynomial $P:\C\longrightarrow\R$ subharmonic, without harmonic terms and $P(0)=0$, such that 
\begin{enumerate}
\item  there exists $\epsilon_n\rightarrow0$ such that 
$r_n:=\frac{1}{\epsilon_n}r\circ\psi_n^{-1}\rightarrow r_P$ locally uniformly on $\C^2$;
\item there exists $A_n\searrow1$ such that 
$$A^{-1}_nM_{r_n}\leq(d\psi_n^*)M_r\leq A_nM_{r_n}$$ in $\psi_n(\Omega\cap U)$ and $M_{r_n}\rightarrow M_{r_P}$ locally uniformly on $\Omega_P$;
\item $\psi_n(\eta_n)\rightarrow(0,-1)$.
\end{enumerate}
\end{theorem}

After a holomorphic change of coordinates we can suppose that $\eta_n$ converges to the origin and 
\begin{equation}\label{tayexp}
r(z,w)=\re[w]+H(z)+o(|z|^{m+1}+|z||w|)
\end{equation}
where $H:\C\longrightarrow\R$ is a homogeneous polynomial of degree $m$, subharmonic and without harmonic terms (see \cite{BP}). Consider $\epsilon_n>0$ such that $\left(\eta^{(1)}_n,\eta^{(2)}_n+\epsilon_n\right)$ is in $\partial\Omega$ and define the following automorphism of $\C^2$
$$\phi_n^{-1}(z,w)=\left(\eta^{(1)}_n+z,\eta^{(2)}_n+\epsilon_n+d_{n,0}w+\sum_{k=2}^md_{n,k}z^k\right) $$ where $d_{n,k}$ are chosen such that $r\circ\phi^{-1}_n=\re[w]+Q_n(z)+o(|z|^{m+1}+|z||w|)$ where $Q_n:\C\longrightarrow\R$ is a subharmonic polynomial without harmonic terms and $Q_n(0)=0$. Notice that for $n\rightarrow\infty$ we have $d_{n,0}\rightarrow1$ and $d_{n,k}\rightarrow0$ for all $k=1,\dots m$.
Now define $\tau_n>0$ such that $$||Q_n(\tau_n\cdot)||=\epsilon_n$$
where $||\cdot||$ is a norm in the space of polynomials of degree less then $m$: if we denote with $P_n(z)=\frac{1}{\epsilon_n}Q_n(\tau_nz)$ there exists a subsequence such that $P_n$ converges to a polynomial  $P$. Finally, we define $$\delta_n(z,w)=\left(\frac{z}{\tau_n},\frac{w}{\epsilon_n}\right),$$
and $\psi_n=\delta_n\circ\phi_n$.

By a simple computation we have 
$$\frac{1}{\epsilon_n}r\circ\psi_n^{-1}\rightarrow r_P $$
uniformly in a neighborhood of the origin in $\mathcal{C}^m$ topology, and that 
$$\psi_n(\eta_n)=\left(0,-d_{n,0}^{-1}\right)\rightarrow (0,-1) $$

Now, we want to prove (2). First of all we have
$$\frac{\partial r_n}{\partial w}(p)=\frac{\partial}{\partial w}\left[\frac{1}{\epsilon_n}r\circ\psi_n^{-1}(p)\right]=d_{0,n}\frac{\partial r}{\partial w}(\psi_n^{-1}(p))$$

$$\frac{\partial r_n}{\partial z}(p)=\frac{\partial}{\partial z}\left[\frac{1}{\epsilon_n}r\circ\psi_n^{-1}(p)\right]=\left[\frac{\tau_n}{\epsilon_n}\frac{\partial r}{\partial z}(\psi_n^{-1}(p))+\frac{\tau_n\sum_{k=2}^mkd_{n,k}(\tau_nz)^{k-1}}{\epsilon_n}\frac{\partial r}{\partial w}(\psi_n^{-1}(p))\right].$$
And this implies that 
\begin{align*}\
\tau_n((d\psi_n)_*L_{r})(p)&=\tau_n(d\psi_n)_{\psi_n^{-1}(p)}L_r(\psi^{-1}_n(p))\\&=\begin{pmatrix}
1 & 0\\
-d_{0,n}^{-1}\tau_n\epsilon_n^{-1}\sum_{k=2}^mkd_{n,k}(\tau_nz)^{k-1} & d_{0,n}^{-1}\tau_n\epsilon_n^{-1}
\end{pmatrix}\begin{pmatrix}
1\\
-\left(\frac{\partial r}{\partial w}(\psi_n^{-1}(p))\right)^{-1}\frac{\partial r}{\partial z}(\psi_n^{-1}(p))
\end{pmatrix}\\&=\begin{pmatrix}
1\\
-\left(\frac{\partial r_n}{\partial w}(p)\right)^{-1}\frac{\partial r_n}{\partial z}(p)
\end{pmatrix}=L_{r_n}(p),
\end{align*}

moreover, $$L_{r_n}\rightarrow L_{r_P}=\frac{\partial}{\partial z}-2P'\frac{\partial}{\partial w}$$.

Now using the chain rule for the Levi form we have
\begin{align*}
\frac{\tau_n^2}{\epsilon_n}\L^{1,1}_{r}({\psi^{-1}_n(p)})&=\frac{\tau_n^2}{\epsilon_n}\L_{r}({\psi^{-1}_n(p)})(L_r(\psi^{-1}_n(p)),\overline{L}_r(\psi^{-1}_n(p)))\\ &=\frac{\tau_n^2}{\epsilon_n}\L_{r\circ\psi^{-1}_n}(p)((d\psi_n)_*L_r(p),\overline{(d\psi_n)_*L_r}(p))\\
&=\L_{\frac{1}{\epsilon_n}r\circ\psi^{-1}_n}(p)(\tau_n(d\psi_n)_*L_r(p),\tau_n\overline{(d\psi_n)_*L_r}(p))\\
&=\L_{r_n}(p)(L_{r_n}(p),\overline{L}_{r_n}(p))=\L^{1,1}_{r_n}(p)
\end{align*}
Consequently, we have for each $l\in\{2,\dots,m\}$
$$\frac{\tau_n^l}{\epsilon_n}C_l^r(\psi^{-1}_n(p))=C_l^{r_n}(p),$$
and 
$$C_l^{r_n}(p)\rightarrow C_l^{r_P}(p)=A_l^P(z)$$
uniformly on compact sets of $\C^2$. 

Now, let $v=\begin{pmatrix}x\\y\end{pmatrix}\in\C^2$ and $a(p,v),b(p,v)\in\C$ such that $v=a(v,p)L_r(p)+b(v,p)N(p)$. Then,
$$a(p,v)=x$$
$$b(p,v)=y+\left(\frac{\partial r}{\partial w}(p)\right)^{-1}\frac{\partial r}{\partial z}(p)x $$
Thus,
$${\tau_n}^{-1}a(\psi_n^{-1}(p),(d\psi_n^{-1})_pv)=x$$
and
$${\epsilon_n}^{-1}b(\psi_n^{-1}(p),(d\psi_n^{-1})_pv)=d_{0,n}\left[y+\left(\frac{\partial r_n}{\partial w}(p)\right)^{-1}\frac{\partial r_n}{\partial z}(p)x\right]$$
that implies the thesis if we set $A_n=\max\{|d_{n,0}|, |d_{n,0}|^{-1}\}\rightarrow1$. 

\begin{remark}\label{rmk1}
From (2) of the previous Theorem we notice that the maps $\psi_n:\Omega\cap U\longrightarrow\psi_n(\Omega\cap U)$ are $(A_n,0)$ quasi-isometries with respect to the Catlin distance, consequently if $\sigma:I\longrightarrow\Omega\cap U$ is a $(A,0)$ quasi-geodesic then $\psi_n\circ\sigma$ is a $(A_nA,0)$ quasi-geodesic in $\psi_n(\Omega\cap U)$. Finally, since $A_n\searrow1$ and $M_{r_n}\to M_{r_p}$, if $\psi_n\circ\sigma$ converges locally uniformly to some function $\hat{\sigma}:I\longrightarrow\Omega_P$ then $\hat{\sigma}$ is a $(A,0)$ quasi-geodesic of $\Omega_P$ with respect to the Catlin distance. 
\end{remark}

\section{Gromov hyperbolicity of model domains}

In this section we prove the Gromov hyperbolicity of the Catlin distance in the model domains. In particular if the polynomial is homogeneous, using Proposition \ref{cathom} we have the Gromov hyperbolicity of the Kobayashi distance.
The strategy is inspired by the proof in \cite{Zim1} of Gromov hyperbolicity of convex finite type domains.

We start with a scaling process for model domains.

\begin{proposition}\label{scalingpol}
Let $\Omega_Q$ be a model domain and let $\{\eta_n\}_{n\in\N}$ be a sequence of points in $\Omega_Q$. Then, passing to a subsequence, there exist a sequence of automorphisms of $\C^2$ $\{\psi_n \}_{n\in\N}$, 
a $\lambda_n>0$ and $P_n:\C\longrightarrow\R$ a subharmonic polynomial without harmonic terms, $P_n(0)=0$ that converges local uniformly to a subharmonic polynomial without harmonic terms $P$, such that
\begin{enumerate}
\item $\lambda_n^{-1}r_P\circ\psi^{-1}_n=r_{P_n}\rightarrow r_P$ locally uniformly on $\C^2$,
\item $(d\psi_n^*)M_{r_Q}=M_{r_{P_n}}\rightarrow M_{r_P}$ locally uniformly on $\Omega_P$,
\item $\psi_n(\eta_n)=(0,-1).$
\end{enumerate}
\end{proposition}

\proof
The proof is very similar to the construction in Section 4. 

We denote $\lambda_n=-r_Q(\eta_n)>0$ and consider the following automorphism of $\C^2$
$$\phi_n^{-1}(z,w)=\left(\eta^{(1)}_n+z,\eta^{(2)}_n+\lambda_n+w+\sum_{k=2}^md_{n,k}z^k\right) $$ where $m$ is the degree of $Q$ and $d_{n,k}$ are chosen so that
$$(r_Q\circ\psi_n^{-1})(z,w)=\re[w]+Q(z+\eta_n^{(1)})-Q(\eta_n^{(1)})+\re\left[\sum_{k=2}^md_{n,k}z^k\right]=:\re[w]+Q_n(z)$$
where $Q_n$ has no harmonic terms. Now we define $\tau_n>0$ such that 
$$||Q_n(\tau_n\cdot)||=\lambda_n$$
where $||\cdot||$ is a norm in the space of polynomials of degree less than $m$, in this way if we set $P_n(z):=\lambda_n^{-1}Q_n(\tau_nz)$ there exists a subsequence such that $P_n$ converges to a polynomial $P$. Again we define
$$\delta_n(z,w)=\left(\frac{z}{\tau_n}, \frac{w}{\lambda_n}\right) $$
and if we set $\psi_n=\delta_n\circ\phi_n$, we obtain by construction (1) and (3). 
For (2), fix $n\in\N$, $p:=(z,w)\in \Omega_{P_n}$ and $v:=(x,y)\in\C^2$. We notice that
$$A^Q_l(\tau_nz+\eta^{(1)}_n)=\lambda_n\tau_n^{-l}A^{P_n}_l(z),$$
then using Lemma \ref{Catmod} we obtain
\begin{align*}
(d\psi_n^*)M_Q(p,v)&=M_Q(\psi_n^{-1}(p),(d\psi_n^{-1})_pv)\\&=\frac{|\lambda_ny+\tau_nx[\sum_{k=0}^mkd_{n,k}(\tau_nz)^{k-1}+2Q'(\tau_nz+\eta_n^{(1)})|}{|r_Q(\psi^{-1}_n(p))|}+\tau_n|x|\sum_{l=2}^{m}\Bigl(\frac{A^Q_l(\tau_nz+\eta_n^{(1)})}{|r_Q(\psi^{-1}_n(p))|}\Bigr)^{\frac{1}{l}}
\\&=\frac{|\lambda_ny+2\lambda_nxP_n'(z)|}{\lambda_n|r_{P_n}(p)|}+\tau_n|x|\sum_{l=2}^{m}\Bigl(\frac{\lambda_n\tau_n^{-l}A^{P_n}_l(z)}{\lambda_n|r_{P_n}(p)|}\Bigr)^{\frac{1}{l}}\\&=M_{P_n}(p,v).
\end{align*}
Finally, since $P_n\rightarrow P$ locally uniformly on $\C$, we have $M_{P_n}\to M_{P}$ locally uniformly on $\Omega_P$.
\endproof

Now we want to study the geodesics in model domains with respect to the Catlin metric.

\begin{lemma}\label{geomodel}
Let $\Omega_P$ be a model domain, then for each $(z,w)\in\partial\Omega_P$ and $a>0$ the following curve
$$\sigma(t)=(z,w-ae^{-t}),\ \ \ t\in\R $$
is a geodesic line with respect to the Catlin metric.
\end{lemma}
\proof
First of all we notice that for each $p,q\in\Omega_P$ and $\gamma$ a generic $\mathcal{C}^1$ piecewise curve joining $p$ to $q$ we have 
\begin{equation}\label{estimcat}
\begin{split}
d^C_{r_P}(p,q)&=\inf_{\gamma}\int M_{r_P}(\gamma(u),\gamma'(u))du\\&\geq\inf_\gamma\int\frac{|\gamma_2'(u)+2\gamma'_1(t)P(\gamma_2(t))|}{-\re[\gamma_2(u)]-P(\gamma_1(u))}du\geq\left|\log\left(\frac{r_P(p)}{r_P(q)}\right)\right|
\end{split}
\end{equation}
so in particular for all $s\leq t$ we obtain $d^C_{r_P}(\sigma(s),\sigma(t))\geq t-s$. Moreover from a simple computation we have
$$length(\sigma|_{[s,t]})=t-s,$$
which implies that $\sigma$ is a geodesic line.
\endproof

In order to talk about the Gromov hyperbolicity of $(\Omega_P,d^C_{r_P})$ we need to know that it is a geodesic space. Thanks to Hopf-Rinow Theorem, it is enough to prove that it is a complete metric space.

\begin{proposition}
Let $\Omega_P$ be a model domain, then $(\Omega_P,d^C_{r_P})$ is a complete metric space, in particular it is a geodesic space.
\end{proposition}
\proof
Fix $o\in\Omega_P$. It is sufficient to show that if $p_n\rightarrow\xi\in\partial\Omega_P\cup\{\infty\}$ then $$d^C_{r_P}(o,p_n)\rightarrow+\infty.$$ Then consider for each $n\in\N$ a piecewise $\mathcal{C}^1$ curve $\gamma_n:[0,1]\longrightarrow\Omega_P$ such that $\gamma_n(0)=o$ and $\gamma(1)=p_n$ and $d^C_{r_p}(o,p_n)+1\geq length(\gamma_n)$. Define
$$M:=\sup_{n\in\N}\max_{t\in [0,1]}|\log(|r_P(\gamma_n(t))|)|.$$
We distinguish two cases: If $M=+\infty$, then there exists $t_n\in[0,1]$ and $q_n:=\gamma_n(t_n)$ such that $|\log(|r_P(q_n)|)|\rightarrow+\infty$ and we obtain using the estimates (\ref{estimcat}) 
$$d^C_{r_p}(o,p_n)+1\geq length(\gamma_n)\geq length(\gamma_n|_{[0,t_n]})\geq d^C_{r_p}(o,q_n)\geq\left|\log\left(\frac{r_P(q_n)}{r_P(o)}\right)\right|\rightarrow+\infty.$$
Otherwise if $M<+\infty$, first of all $p_{1,n}\rightarrow\infty$ and we have that for each $n\in\N$ and $t\in[0,1]$
$$e^{-M}\leq|r_P(\gamma_n(t))|\leq e^M.$$
Let $m$ be the degree of $P(z):=\sum_{j,k>0,j+k\leq m}a_{j,k}z^j\bar{z}^k$, first of all we notice that $$A^P_m(z)=\max\{j!k!|a_{j,k}|:j,k>0,j+k=m \}=:A>0$$
does not depend on $z\in\C$. Then
$$M_{r_P}(\gamma_n(t),\gamma_n'(t))\geq \frac{A|\gamma_{1,n}'(t)|}{|r_P(\gamma_n(t))|^\frac{1}{m}}\geq Ae^{-\frac{M}{m}}|\gamma_{1,n}'(t)| $$
that implies
$$length(\gamma_n)\geq Ae^{-\frac{M}{m}}|p_{1,n}-o_1|\rightarrow+\infty.$$\endproof

Now we have a sequence of results concerning the behavior of the geodesics.

\begin{proposition}\label{prop1model}
Let $\Omega_P$ be a model domain, let $p_n, q_n\in\Omega_P$ be two sequences with $p_n\rightarrow\xi^+\in\partial\Omega\cup\{\infty\}$ and $q_n\rightarrow\xi^-\in\partial\Omega$, and 
$$\liminf_{n,m\rightarrow\infty}(p_n|q_m)_o<\infty,$$
then $\xi^+\neq\xi^-$ (the Gromov product is with respect to $d_{r_P}^C$).
\end{proposition}
\proof
By passing to subsequences we may assume $\lim_{n\rightarrow\infty}(p_n|q_n)_o$  exists and it is finite. Assume by contradiction that $\xi^+=\xi^-=:\xi\in\partial\Omega$.
Define $a_n=-r_P(p_n)>0$ and $b_n=-r_P(q_n)>0$ and 
$$\sigma^+_n(t)=p_n+(0,a_n-e^{-t}),\ \ \mbox{and}\ \ \sigma^-_n(t)=q_n+(0,b_n-e^{-t}).$$
By Lemma \ref{geomodel}, they are geodesics. Now notice that there exists $R>0$ such that for each $n$ we have $d^C_{r_P}(\sigma^+_n(0),o)<R$ and $d^C_{r_P}(\sigma^-_n(0),o)<R$. Finally, $\sigma^+_n(-\ln(a_n))=p_n$ and $\sigma^+_n(-\ln(b_n))=q_n$. 

Now
$$d^C_{r_P}(o,p_n)\geq d^C_{r_P}(\sigma^+_n(0),p_n)-d^C_{r_P}(o,\sigma^+_n(0))\geq -\ln(a_n)-R$$ 
and similarly 
$$d^C_{r_P}(o,q_n)\geq d^C_{r_P}(\sigma^-_n(0),q_n)-d^C_{r_P}(o,\sigma^-_n(0))\geq -\ln(b_n)-R.$$
Instead, fixing $T>0$,
\begin{align*}d^C_{r_P}(p_n,q_n)&\leq d^C_{r_P}(p_n,\sigma_n^+(T))+d^C_{r_P}(\sigma_n^+(T),\sigma_n^+(T))+d^C_{r_P}(q_n,\sigma_n^-(T))\\&=-2T-\ln(a_nb_n)+d^C_{r_P}(\sigma_n^+(T),\sigma_n^+(T)).
\end{align*}
So
$$
2(p_n|q_n)_o=d^C_{r_P}(o,p_n)+d^C_{r_P}(0,q_n)-d^C_{r_P}(p_n,q_n)\geq\\
-2T-2R-d^C_{r_P}(\sigma_n^+(T),\sigma_n^+(T))
$$
but 
\begin{align*}\lim_{n\rightarrow\infty} d^C_{r_P}(\sigma_n^+(T),\sigma_n^+(T))&=\lim_{n\rightarrow\infty}d^C_{r_P}(p_n+(0,a_n-e^{-T}),q_n+(0,b_n-e^{-T}))\\&=d^C_{r_P}(\xi+(0,-e^{-T}),\xi+(0,-e^{-T}))=0
\end{align*}
Which implies 
$$\lim_{n\rightarrow\infty}(p_n|q_n)_o\geq T-R $$
by the arbitrariness of $T$ we have a contradiction.  
\endproof

Now we need the following uniform estimate on the Catlin metric (note the analogy with the concept of Goldilocks domains in \cite{BZ}).

\begin{lemma}\label{unifgold}
Let $\Omega$, $\{\psi_n\}_{n\in\N}$ and $P$ be as in Theorem \ref{scalingth} or \ref{scalingpol}, then for each $R>0$ there exist $c>0$ and $C>0$ such that eventually in $n$
$$M_{r_n}(x,v)\geq\frac{c||v||}{|r_n(x)|^\frac{1}{m}}, \ \ \forall x\in\psi_n(\Omega)\cap B_{R}(0), v\in\C^2 $$
and
$$d^C_{r_n}(x,o)\leq C+\ln\left(\frac{1}{|r_n(x)|}\right), \ \ \forall x\in\psi_n(\Omega)\cap B_{R}(0).$$
\end{lemma}
\proof
First of all notice that since $p\in B_R(0)$ and $r_n\rightarrow r_P$ locally uniformly there exists $A>0$ such that  $|r_n(p)|\leq A$ for each $n\in\N$,
moreover for each $l\in\{1,\dots,m \}$
$$|r_n(p)|^{\frac{1}{l}}\leq A^{\frac{1}{l}-\frac{1}{m}}|r_n(p)|^{\frac{1}{m}}.$$
Now since there exists $B>0$ such that for each $p=(z,w)\in B_R(0)$ 
$\sum_{l=2}^m(A_l^{P}(z))^{\frac{1}{l}}>B$ and 
$$\sum_{l=2}^m(C_l^{r_n}(p))^{\frac{1}{l}}\rightarrow\sum_{l=2}^m(A_l^{P}(z))^{\frac{1}{l}}, $$ then eventually in $n$
$$\sum_{l=2}^m(C_l^{r_n}(p))^{\frac{1}{l}}\geq B. $$
Using the same argument there exists $D>0$ such that for each $p\in B_R(0)$ we have $$\left|\left(\frac{\partial r_n}{\partial w}(p)\right)^{-1}\frac{\partial r_n}{\partial z}(p)\right|\leq D$$
Moreover, we notice that for each $x,y,a\in\C$ we have 
$$|y+ax|+|x|\geq\frac{1}{1+|a|}(|x|+|y|).$$
Finally,
\begin{align*}
M_{r_n}(p,(x,y))&=\frac{|y+x\left(\frac{\partial r_n}{\partial w}(p)\right)^{-1}\frac{\partial r_n}{\partial z}(p)|}{|r_n(p)|}+|x|\sum_{l=2}^{m}\Bigl(\frac{C^{r_n}_l(z)}{|r_n(p)|}\Bigr)^{\frac{1}{l}}\\&\geq \frac{1}{|r_n(p)|^{\frac{1}{m}}}\left(A^{\frac{1}{m}-1}\left|y+x\left(\frac{\partial r_n}{\partial w}(p)\right)^{-1}\frac{\partial r_n}{\partial z}(p)\right|+|y|BA^{\frac{1}{m}}\right)\\&\geq\frac{A^{\frac{1}{m}}\min\{A^{-1},B\}}{1+D}\frac{|x|+|y|}{|r_n(p)|^\frac{1}{m}}\geq \frac{c||v||}{|r_n(p)|^{\frac{1}{m}}} 
\end{align*}
that it is first estimate if we set $c:=\frac{A^{\frac{1}{m}}\min\{A^{-1},B\}}{1+D}$.

For the second estimate, we notice that for each $R'>0$ there exists $C(R')>0$ such that 
$$d^C_{r_P}(q,o)< C(R'), \ \ \forall q\in B_{R'}(0), r_P(q)=-1$$
and, since $d^C_{r_n}$ converges uniformly to $d^C_{r_P}$ on the compact sets of $\Omega\times\Omega$, we have 
$$d^C_{r_n}(q,o)\leq C(R'), \ \ \forall q\in B_{R'}(0), r_n(q)=r_n(o).$$
Now for each $p\in\psi_n(\Omega)\cap B_R(0)$, let $q=p+(0,\lambda)\in\psi_n(\Omega)$ with $\lambda\in\R$ be a point such that $r_n(q)=r_n(o)$, then using Lemma \ref{geomodel} we have that $$d^C_{r_n}(p,q)=\left|\ln\left(\frac{r_n(q)}{r_n(p)}\right)\right|=\left|\ln\left(\frac{r_n(o)}{r_n(p)}\right)\right|.$$
Now notice that there exists $R'>0$ such that for each $n\in\N$ and $p\in\psi_n(\Omega)$ the corresponding $q$ is in $B_{R'}(0)$, then 
\begin{align*}
d^C_{r_n}(p,o)\leq d^C_{r_n}(p,q)+d^C_{r_n}(q,o)&\leq \left|\ln\left(\frac{r_n(o)}{r_n(p)}\right)\right|+A(R')\\&\leq \ln\left(\frac{1}{|r_n(p)|}\right)+\ln(A)+C(R')\leq \ln\left(\frac{1}{|r_n(p)|}\right)+C
\end{align*}
that it is the second estimate if we set $C:=\ln(A)+C(R')$.
\endproof

\begin{corollary}\label{uniflip}
Let $\Omega$, $\{\psi_n\}_{n\in\N}$ and $P$ be as in Theorem \ref{scalingth} or \ref{scalingpol}, then for each $R>0$ and $A\geq1$ there exists a $L:=L(R,A)>0$ such that for every $n\in\N$ a $(A,0)$ quasi-geodesic $\sigma$ of $\psi_n(\Omega)$ contained in $B_R(0)$ is $L$-Lipschitz (with respect to the Euclidean distance of $\C^2$) and
$$M_{r_n}(\sigma(t),\sigma'(t))\leq A$$
for almost every $t\in [0,T]$.
\end{corollary}
\proof
For each $0\leq s\leq t\leq T$, there exists $\gamma_{s,t}:[0,1]\longrightarrow\psi_n(\Omega)$ a piecewise $\mathcal{C}^1$ curve such that $\gamma_{s,t}(0)=\sigma(s)$, $\gamma_{s,t}(1)=\sigma(t)$ and $d^C_{r_P}(\sigma(s),\sigma(t))\geq \frac{1}{2}length(\gamma_{s,t}) $.
Now there exists $R>0$ such that for every $s,t$ we have $\gamma_{s,t}([0,1])\subset B_R(0)$, therefore by Lemma \ref{unifgold} there exists $c>0$ such that for every $u\in[s,t]$
$$M_{r_n}(\gamma_{s,t}(u),\gamma_{s,t}'(u))\geq c||\gamma'_{s,t}(u)||$$
that implies
$$length(\gamma_{s,t})\geq c||\sigma(s)-\sigma(t)||$$
and so
$$||\sigma(s)-\sigma(t)||\leq 2c^{-1}d^C_{r_P}(\sigma(s),\sigma(t))\leq 2c^{-1}A|t-s|. $$
Now for each $t\in [0,T]$ and $h>0$ such that $t+h\in[0,T]$, we have from \cite[Theorem 1.2]{Vent}
$$\int_{t}^{t+h}M_{r_n}(\sigma(u),\sigma'(u))du=\sup_{\mathcal{P}}\sum_{j=0}^{N(\mathcal{P})}d_{r_n}^C(\sigma(u_j),\sigma(u_{j+1}) )$$
where the supremum is taken over all the partitions $\mathcal{P}$ of $[t,t+h]$. Finally, since $\sigma$ is a $(A,0)$ quasi-geodesic
$$\int_{t}^{t+h}M_{r_P}(\sigma(u),\sigma'(u))du\leq Ah $$
and so using the Lebesque differentiation theorem
$$M_{r_n}(\sigma(t),\sigma'(t))\leq A$$
for almost every $t\in [0,T]$.
Using the first estimate of Lemma \ref{unifgold} and noticing that there exists $D>0$ such that $|r_n(p)|\leq D$ for each $n\in\N$ and $p\in B_R(0)$, we have
$$A\geq M_{r_n}(\sigma(t),\sigma'(t))\geq \frac{c||\sigma'(t)||}{|r_n(\sigma(t))|^{\frac{1}{m}}}\geq cD^{-\frac{1}{m}}||\sigma'(t)|| $$
that implies
$$||\sigma'(t)||\leq AD^{\frac{1}{m}}c^{-1}=:L $$
thus $\sigma$ is $L$-Lipschitz.
\endproof

\begin{proposition}\label{visibscal}
Let $\Omega\subset\C^2$, $\{\psi_n\}_{n\in\N}$ and $P:\C\longrightarrow\R$ be as in Theorem \ref{scalingth} or \ref{scalingpol}. Let $\sigma_n:[a_n,b_n]\longrightarrow\Omega$ be a sequence of $(A,0)$ quasi-geodesics with respect to the Catlin metric, and define $\tilde{\sigma}_n:=\psi_n\circ\sigma_n$. Suppose that there exists $R>0$ such that
\begin{enumerate}
	\item $|b_n-a_n|\to\infty$;
	\item $\tilde{\sigma}_n([a_n,b_n])\subset B_R(0)$;
	\item $\lim_{n\rightarrow\infty}||\tilde{\sigma}_n(a_n)-\tilde{\sigma}_n(b_n)||>0$,
\end{enumerate}
then, passing to a subsequence, there exists $T_n\in[a_n,b_n]$ such that the sequence of $t\mapsto\tilde{\sigma}_n(t+T_n)$ converges uniformly on compact set to a $(A,0)$ quasi-geodesic $\tilde{\sigma}:\R\longrightarrow\Omega_{P}$.
\end{proposition}
\proof
The argument is very similar to the proof of Theorem 1.4 in \cite{BZ}.

Suppose
$$\lim_{n\rightarrow\infty}||\tilde{\sigma}_n(a_n)-\tilde{\sigma}_n(b_n)||=\epsilon,$$
then there exists $T_n\in[a_n,b_n]$ such that 
$$\lim_{n\rightarrow\infty}||\tilde{\sigma}_n(a_n)-\tilde{\sigma}_n(T_n)||=\lim_{n\rightarrow\infty}||\tilde{\sigma}_n(T_n)-\tilde{\sigma}_n(b_n)||\geq\epsilon\slash2. $$
By passing to a subsequence we may assume that $\tilde{\sigma}_n(T_n)$ converges to a point $p\in\overline{\Omega_{P}\cap B_R(0)}$. If $p\in\Omega_{P}$, then by Ascoli-Arzelà theorem there exists a subsequence such that $\tilde{\sigma}_n$ converges to a $(A,0)$ quasi-geodesic $\tilde{\sigma}:\R\longrightarrow\Omega_{P}$ (Remark \ref{rmk1}).

Otherwise if $p\in\partial\Omega_{P}$, we can suppose that $\tilde{\sigma}_n(a_n)\rightarrow\xi^-\in\overline{\Omega_{P}\cap B_R(0)}$ and $\tilde{\sigma}_n(b_n)\rightarrow\xi^+\in\overline{\Omega_{P}\cap B_R(0)}$: if  both $\xi^+,\xi^-\in\Omega_{P}$ then
\begin{align*}
d^C_{\Omega_P}(\xi^+,\xi^-)&=\lim_{n\rightarrow\infty}d^C_{r_n}(\tilde{\sigma}_n(a_n),\tilde{\sigma}_n(b_n))\\
&\geq\lim_{n\rightarrow\infty}A^{-1}|b_n-a_n|\\&
=\lim_{n\rightarrow\infty}A^{-1}(|b_n-T_n|+|T_n-a_n|)
\\&\geq\lim_{n\rightarrow\infty}A^{-2}d^C_{r_n}(\sigma_n(b_n),\sigma_n(T_n))+A^{-2}d^C_{r_n}(\sigma_n(T_n),\sigma_n(a_n)))=\infty
\end{align*}
that is impossible. Then one of $\xi^+,\xi^-$ is in $\partial\Omega_{P}$, for example $\xi^+=:\xi\in\partial\Omega_{P}$.

Now for $n\rightarrow\infty$
$$\max\{|r_{n}(\tilde{\sigma}_n(t))|: t\in[T_n,b_n]\}\rightarrow0$$
indeed if  there exists $T'_n\in[T_n,b_n]$ such that $|r_{n}(\tilde{\sigma}_n(T'_n))|>C>0$ then using Ascoli-Arzelà theorem we have a contradiction. After a change of parametrization we can suppose that $0\in[T_n,b_n]$ and
$$|r_{n}(\tilde{\sigma}_n(0))|=\max\{|r_{n}(\tilde{\sigma}_n(t))|: t\in[T_n,b_n]\}\rightarrow0.$$

Now by Lemma \ref{uniflip}, $\tilde{\sigma}_n|_{[T_n,b_n]}$ are $L$-Lipschitz, then there exists a subsequence that converges uniformly on compact sets to $\hat{\sigma}:\R\longrightarrow\partial\Omega_{P}$. We prove the following two Claims

\textbf{Claim 1:} $\hat{\sigma}$ is constant. Indeed  we have from Lemma \ref{unifgold} that for each $t\in[T_n,b_n]$
$$A\geq M_{r_n}(\tilde{\sigma}_n(t),\tilde{\sigma}'_n(t))\geq\frac{c||\tilde{\sigma}_n'(t)||}{|r_{n}(\tilde{\sigma}_n(t))|} $$
consequently 
$$||\tilde{\sigma}'_n(t)||\leq Ac^{-1}|r_{n}(\tilde{\sigma}_n(t))| $$
for each $u,v\in\R$
$$||\hat{\sigma}(u)-\hat{\sigma}(v)||=\int_u^v ||\hat{\sigma}'(t)||dt=\lim_{n\rightarrow\infty}\int_u^v ||\tilde{\sigma}_n'(t)||dt\leq\lim_{n\rightarrow\infty}Ac^{-1}\int_u^v |r_{n}(\tilde{\sigma}_n(t))|dt=0,$$
i.e. $\hat{\sigma}$ is constant.

\textbf{Claim 2:} $\hat{\sigma}$ is not constant. We have from the second estimate of Lemma \ref{unifgold}
$$A^{-1}|t|\leq d^C_{r_n}(\tilde{\sigma}_n(0),\tilde{\sigma}_n(t))\leq d^C_{r_n}(\tilde{\sigma}_n(0),o)+d^C_{r_n}(o,\tilde{\sigma}_n(t))\leq 2C +\ln\left(\frac{1}{|r_{n}(\tilde{\sigma}_n(0))r_{n}(\tilde{\sigma}_n(t))|}\right) $$
Then
$$|r_{n}(\tilde{\sigma}_n(t))|\leq\sqrt{|r_{n}(\tilde{\sigma}_n(t))r_{n}(\tilde{\sigma}_n(0))|}\leq e^{C-\frac{A}{2}|t|}.$$
Now let $T',b'\in\R$
\begin{align*}
||\hat{\sigma}(T')-\hat{\sigma}(b')||&=\lim_{n\rightarrow\infty}||\tilde{\sigma}_n(T')-\tilde{\sigma}(b')||\\
&\geq\lim_{n\rightarrow\infty}(||\tilde{\sigma}_n(b_n)-\tilde{\sigma}_n(T_n)||-||\tilde{\sigma}_n(T')-\tilde{\sigma}_n(T_n)||-||\tilde{\sigma}_n(b_n)-\tilde{\sigma}_n(b')||)\\&\geq||p-\xi||-\limsup_{n\rightarrow\infty}\int_{b'}^{b_n}||\tilde{\sigma}'_n(t)||dt-\limsup_{n\rightarrow\infty}\int_{T_n}^{T'}||\tilde{\sigma}'_n(t)||dt\\
&\geq||p-\xi||-Ac^{-1}\limsup_{n\rightarrow\infty}\int_{b'}^{b_n}|r_{n}(\tilde{\sigma}_n(t))|dt-Ac^{-1}\limsup_{n\rightarrow\infty}\int_{T_n}^{T'}|r_{n}(\tilde{\sigma}_n(t))|dt\\
&\geq ||p-\xi||-Ac^{-1}\limsup_{n\rightarrow\infty}\int_{b'}^{b_n}e^{C-\frac{1}{2}|t|}dt-Ac^{-1}\limsup_{n\rightarrow\infty}\int_{T_n}^{T'}e^{C-\frac{1}{2}|t|}dt\\
&\geq||p-\xi||-Ac^{-1}\int_{b'}^{\infty}e^{C-\frac{A}{2}|t|}dt-Ac^{-1}\int_{-\infty}^{T'}e^{C-\frac{A}{2}|t|}dt>0
\end{align*}
if $-T'$ and $b'$ are big enough, then $\hat{\sigma}$ is not constant.

Claim 1 and Claim 2 are in contradiction, then $p\in\Omega_{P}$. 
\endproof

Now we have several results about the behavior of $(A,0)$ quasi-geodesics, the proofs are very similar to the analogous results in Section 11 in \cite{Zim1}.

\begin{corollary}\label{limgeo}
Let $\Omega\subset\C^2$, $\{\psi_n\}_{n\in\N}$ and $P:\C\longrightarrow\R$ be as in Theorem \ref{scalingth} or \ref{scalingpol}. Let $\sigma_n:[a_n,b_n]\longrightarrow\Omega$ be a sequence of $(A,0)$ quasi-geodesics with respect to the Catlin metric, and define $\tilde{\sigma}_n:=\psi_n\circ\sigma_n$. Suppose that $\sigma_n$ converges locally uniformly to a $(A,0)$ quasi-geodesic line $\tilde{\sigma}:\R\longrightarrow\Omega_P$. If $\lim_{n\rightarrow\infty}\tilde{\sigma}_n(b_n)=:x_\infty$ then
$$\lim_{t\rightarrow+\infty}\tilde{\sigma}(t)=x_\infty$$
\end{corollary}
\proof
By contradiction, suppose there exists $t_n\nearrow+\infty$ such that $$y_\infty:=\lim_{n\rightarrow\infty}\tilde{\sigma}(t_n)\neq x_\infty.$$
Since $\tilde{\sigma}_n$ converges locally uniformly to $\hat{\sigma}$, there exists a subsequence $\{t_{n_k}\}_{k\in\N}$ with $t_{n_k}\in[a_{n_k},b_{n_k}]$ such that 
$$\lim_{k\rightarrow\infty}\tilde{\sigma}_{n_k}(t_{n_k})=y_\infty.$$
Since $x_\infty$ and $y_\infty$ are different, at least one is finite. Hence there exist $t'_{n_k}, b'_{n_k}\in\R$ with $t_{n_k}\leq t'_{n_k}\leq b'_{n_k}\leq b_{n_k}$, $R>0$ and $\epsilon>0$ such that
\begin{enumerate}
	\item $\tilde{\sigma}_{n_k}([t'_{n_k},b'_{n_k}])\subset B_R(0)$,
	\item $\lim_{k\rightarrow\infty}||\tilde{\sigma}_{n_k}(t'_{n_k})-\tilde{\sigma}_{n_k}(b'_{n_k})||>\epsilon.$
\end{enumerate}
We can use Proposition \ref{visibscal}, then there exists $T_k\in[t'_{n_k},b'_{n_k}]$ such that $t\mapsto\tilde{\sigma}_{n_k}(t+T_k)$ converges to a $(A,0)$ quasi-geodesic line $\hat{\sigma}$. Finally
$$d^C_{r_P}(\tilde{\sigma}(0),\hat{\sigma}(0))=\lim_{k\rightarrow\infty}d^C_{r_{n_k}}(\tilde{\sigma}_{n_k}(0),\tilde{\sigma}_{n_k}(T_k))\geq\lim_{k\rightarrow\infty} A^{-1}T_k\geq\lim_{k\rightarrow\infty}A^{-1}t_{n_k}=\infty $$
but this is a contradiction. 
\endproof

\begin{corollary}\label{gbmodel}
Let $\Omega_P$ be a model domain, let $p_n, q_n\in\Omega_P$ be two sequences with $p_n\rightarrow\xi^+\in\partial\Omega\cup\{\infty\}$ and $q_n\rightarrow\xi^-\in\partial\Omega\cup\{\infty\}$, and 
$$\lim_{n,m\rightarrow\infty}(p_n|q_m)_o=\infty,$$
then $\xi^+=\xi^-$.
\end{corollary}
\proof
Suppose by contradiction that $\xi^+\neq\xi^-$ and consider $\sigma_n:[a_n,b_n]\longrightarrow\Omega_P$ be a sequence of geodesics such that $0\in[a_n,b_n]$ and $\sigma_n(a_n)=p_n$ and $\sigma_n(b_n)=q_n$. Now we can use Proposition \ref{visibscal}, then $\sigma_n$ converges local uniformly to a geodesic $\sigma:\R\longrightarrow\Omega_P$. Finally, recalling 
$$|(x|y)_o-(x|y)_{o'}|\leq d(o,o') $$
and, given a geodesic $\gamma$, we have for each $s\leq0\leq t$
$$(\gamma(s)|\gamma(t))_{\gamma(0)}=0.$$
Finally
\begin{align*}\lim_{n\rightarrow\infty}(p_n|q_n)_o&=\lim_{n\rightarrow\infty}
(\sigma_n(a_n)|\sigma_n(b_n))_o\\&\leq\lim_{n\rightarrow\infty}[
(\sigma_n(a_n)|\sigma_n(b_n))_{\sigma_n(0)}+d(\sigma_n(0),o)]=d(\sigma(0),o)<\infty.
\end{align*} 
\endproof

\begin{lemma}\label{geoinf}
Let $\Omega_Q$ be a model domain. Then there exists no geodesic line $\sigma:\R\longrightarrow\Omega_Q$ such that 
$$\lim_{t\rightarrow\infty}\sigma(t)=\lim_{t\rightarrow-\infty}\sigma(t)=\infty $$
\end{lemma}
\proof
We consider $\{\psi_n\}_{n\in\N}$ and $P:\C\longrightarrow \R$ of Proposition \ref{scalingpol} with respect to the sequence $(0,-n)$. Define $\sigma_n=\Phi_n\circ\sigma$, $\sigma_n^+=\sigma_n|_{[0,\infty)}$ and $\sigma_n^-=\sigma_n|_{(-\infty,0]}$. 
By Proposition \ref{visibscal}, there exist two sequences $a_n,b_n\rightarrow\infty$ such that $\sigma^+_n(t-a_n)$ and $\sigma_n^-(t+b_n)$ converge locally uniformly respectively to the geodesics $\sigma^+_\infty:\R\longrightarrow\Omega_{P}$ and $\sigma^-_\infty:\R\longrightarrow\Omega_{P}$. Now 
$$d^C_{r_P}(\sigma^+_\infty(0),\sigma^-_\infty(0))=\lim_{n\rightarrow\infty}d^C_{r_{P_n}}(\sigma^+_n(-a_n),\sigma_n^-(b_n))=\lim_{n\rightarrow\infty}d^C_{r_{Q}}(\sigma(-a_n),\sigma(b_n))=\lim_{n\rightarrow\infty}b_n+a_n=\infty$$
which is a contradiction.
\endproof

Finally we can prove the good behavior of the geodesic lines in the model domains.

\begin{proposition}\label{behgeomode}
Let $\Omega_P$ be a model domain and $\sigma:\R\longrightarrow\Omega_P$ a geodesic line. Then both 
$$\lim_{t\rightarrow\infty}\sigma(t),\ \ \lim_{t\rightarrow-\infty}\sigma(t)$$
exist in $\overline{\C^2}$ and they are different.
\end{proposition}
\proof
First of all we prove that $\lim_{t\rightarrow\infty}\sigma(t)$ exists (the limit to $-\infty$ is analogous). By contradiction, if there exist $t_k,s_k\rightarrow\infty$ such that 
$$\lim_{k\rightarrow\infty}\sigma(t_k)\neq\lim_{k\rightarrow\infty}\sigma(s_k)$$
then
$$\lim_{k\rightarrow\infty}(\sigma(t_k)|\sigma(s_k))_{\sigma(0)}=\lim_{k\rightarrow\infty}\frac{1}{2}[t_k+s_k-|t_k-s_k|]=\lim_{k\rightarrow\infty}\min\{t_k,s_k\}=\infty$$
that implies, using Corollary \ref{gbmodel}, that $\lim_{k\rightarrow\infty}\sigma(t_k)=\lim_{k\rightarrow\infty}\sigma(s_k)$, which is a contradiction. 

Now we have to prove that $\lim_{t\rightarrow\infty}\sigma(t)$ and $\lim_{t\rightarrow-\infty}\sigma(t)$ are different. By Lemma \ref{geoinf} they cannot be both $\infty$, then we can use Proposition \ref{prop1model} after noticing that
$$(\sigma(-t)|\sigma(t))_{\sigma(0)}\equiv0.$$
Then $\lim_{t\rightarrow\infty}\sigma(t)\neq\lim_{t\rightarrow-\infty}\sigma(t)$.
\endproof

Finally we have all the ingredients to prove the main result of this section.

\begin{theorem}\label{mainthmod}
Let $\Omega_Q$ be a model domain, then $(\Omega_Q,d^C_{r_Q})$ is Gromov hyperbolic. Moreover, the identity map $\Omega_Q\longrightarrow\Omega_Q$
extends to a homeomorphism $\overline{\Omega}^G_Q\longrightarrow\overline{\Omega}^*_Q$, where $\overline{\Omega}^*_Q:=\overline{\Omega}_Q\cup\{\infty\}$ is the one point compactification of the Euclidean closure of $\Omega_Q$.
\end{theorem}
\proof
\textbf{Part 1: $(\Omega_Q,d^C_{r_Q})$ is Gromov hyperbolic}

By contradiction, suppose that $(\Omega_Q,d^C_{r_Q})$ is not Gromov hyperbolic, then there exist three sequences of points $\{x_n\}_{n\in\N},\{y_n\}_{n\in\N},\{z_n\}_{n\in\N}$, correspondent geodesic segments $\{\sigma_{x_n,y_n}\}_{n\in\N}$, $\{\sigma_{y_n,z_n}\}_{n\in\N}$, $\{\sigma_{x_n,z_n}\}_{n\in\N}$ and a point $u_n\in \sigma_{x_n,y_n}$ such that
$$\lim_{n\rightarrow\infty}d^C_{\Omega_Q}(u_n,\sigma_{x_n,z_n}\cup\sigma_{y_n,z_n})=\infty.$$
Up to passing to a subsequences, we can suppose that  $u_n\rightarrow\xi\in\overline{\Omega}_Q$. We have two cases

\textbf{Case 1:} $\xi\in\Omega_Q$.

First of all we can suppose that $x_n, y_n,z_n$ converge respectively to $x_\infty,y_\infty,z_\infty\in\overline{\Omega}_Q\cup\{\infty\}$, and since
$$\lim_{n\rightarrow\infty} d^C_{r_Q}(u_n,\{x_n,y_n,z_n\})\geq\lim_{n\rightarrow\infty} d^C_{r_Q}(u_n,\sigma_{x_n,z_n}\cup\sigma_{y_n,z_n})=\infty$$
we have $x_\infty,y_\infty,z_\infty\in\partial\Omega_Q\cup\{\infty\}$.
After a reparametrisation, we can suppose that $\sigma_{x_n,y_n}(0)=u_n$, and using Ascoli-Arzelà Theorem $\sigma_{x_n,y_n}$ converges uniformly on compact sets to a geodesic line $\sigma:\R\longrightarrow\Omega_Q$. Now by Corollary \ref{limgeo} and Proposition \ref{behgeomode} we have $x_\infty\neq y_\infty$, then one of the two is different from $z_\infty$, for example $x_\infty\neq z_\infty$. Using Proposition \ref{visibscal}, there exists $T_n\in\R$ such that  $t\mapsto\sigma_{x_n,z_n}(t+T_n)$ converges locally uniformly to a geodesic line $\hat{\sigma}$. Finally
$$\lim_{n\rightarrow\infty}d^C_{r_Q}(\sigma(0),\hat{\sigma}(0))=\lim_{n\rightarrow\infty}d^C_{r_Q}(\sigma_{x_n,y_n}(0),\sigma_{x_n,z_n}(T_n))\geq \lim_{n\rightarrow\infty}d^C_{r_Q}(u_n,\sigma_{x_n,z_n}\cup\sigma_{y_n,z_n})=\infty$$
and this is a contradiction.

\textbf{Case 2:} $\xi\in\partial\Omega$.

Let $\{\psi_n\}_{n\in\N}$ and $P$ be as in Theorem \ref{scalingth} or \ref{scalingpol} with respect to the sequence $u_n$, then
up to passing to subsequences $\psi_n(x_n), \psi_n(y_n)$ and $\psi_n(z_n)$ converge respectively to $x_\infty, y_\infty$ and $z_\infty$ in $\overline{\Omega_P}\cup\{\infty\}$. Now after a parametrization we can suppose that $\sigma_{x_n,y_n}(0)=u_n$, so using Ascoli-Arzelà $\psi_n\circ\sigma_{x_n,y_n}$ converges to a geodedic $\sigma$ of $\Omega_P$ such that 
$$\lim_{t\rightarrow-\infty}\sigma(t)=x_\infty,\ \mbox{and}\, \lim_{t\rightarrow\infty}\sigma(t)=y_\infty.$$ By Lemma \ref{behgeomode}, $x_\infty\neq y_\infty$ so $z_\infty$ does not equal at least one of the two, for example $x_\infty\neq z_\infty$. Now by Proposition \ref{visibscal}, $\psi_n\circ\sigma_{x_n,z_n}$ converges locally uniformly to the geodesic $\hat{\sigma}:\R\longrightarrow\Omega_Q$. Finally
\begin{align*}
d^C_{r_P}((0,-1),\hat{\sigma}(0))&=\lim_{n\rightarrow\infty}d^C_{r_{P_n}}(\psi_n(u_n),\psi_n\circ\sigma_{x_n,z_n}(0))\\&=\lim_{n\rightarrow\infty}d^C_{r_Q}(u_n,\sigma_{x_n,z_n}(0))\\&\geq\lim_{n\rightarrow\infty}d^C_{r_Q}(u_n,\sigma_{x_n,z_n}\cup\sigma_{y_n,z_n})=\infty
\end{align*}
which is a contradiction. 

Then $(\Omega_Q, d^C_{r_Q})$ is Gromov hyperbolic.

\textbf{Part 2: The identity map $\Omega_Q\longrightarrow\Omega_Q$
	extends to a homeomorphism $\overline{\Omega}^G_Q\longrightarrow\overline{\Omega}^*_Q$.}
Since the proof is essentially identical to the one of Proposition 1.2 in \cite{Zim1} and Part 2 of Theorem \ref{mainth} (see Section 6), we omit it.
\endproof

From the general theory on Gromov hyperbolic spaces \cite[Chapter 5]{GdlH}, also an $(A,B)$ quasi-geodesic ray $\sigma$, too, has a "good" behavior in the Gromov boundary, then we have directly the following

\begin{corollary}\label{behqgeomode}
Let $\Omega_P$ be a model domain and $\sigma:\R\longrightarrow\Omega_P$ be a $(A,B)$ quasi-geodesic line with respect to the Catlin metric. Then both 
$$\lim_{t\rightarrow\infty}\sigma(t),\ \ \lim_{t\rightarrow-\infty}\sigma(t)$$
exist in $\overline{\C^2}$ and they are different.
\end{corollary}

Finally, using the fact that the Catlin and Kobayashi metric are bi-Lipschitz in homogeneous model domains (Proposition \ref{cathom}) we obtain the following

\begin{corollary}\label{Gromohom}
Let $\Omega_H$ be a model domain where $H:\C\longrightarrow \R$ is a subharmonic homogeneous polynomial without harmonic terms, then $(\Omega_H,d^K_{\Omega_H})$ is Gromov hyperbolic. Moreover, the identity map $\Omega_H\longrightarrow\Omega_H$
extends to a homeomorphism $\overline{\Omega}^G_H\longrightarrow\overline{\Omega}^*_H$, where $\overline{\Omega}^*_H:=\overline{\Omega}_H\cup\{\infty\}$ is the one point compactification of the Euclidean closure of $\Omega_H$.
\end{corollary}

\subsection{Gromov hypebolicity of domains with non-compact automorphism group}

At this point Theorem \ref{ncautoth} descends easily from the results of this section: the proof is direct consequence of  Corollary \ref{Gromohom} and the following result by Berteloot.

\begin{theorem}\cite{Bert3}
Let $\Omega\subset\C^2$ be a pseudoconvex domain. Suppose that there exists $\xi\in\partial\Omega$ such that $\partial\Omega$ is smooth and of finite type in a neighborhood of $\xi$ and there exists a sequence of automorphisms $\{\phi_k\}_{k\in\N}$ and a point $p\in\Omega$ such that $\phi_k(p)\rightarrow\xi$. Then there exists a subharmonic homogeneous polynomial $H:\C\longrightarrow\R$ without harmonic terms such that $\Omega$ is biholomorphic to $\Omega_H$.
\end{theorem}

\section{Pseudoconvex bounded smooth finite type domains are Gromov hyperbolic}

Now we want to prove our main result (Theorem \ref{mainth}). As in the previous section, we start studying the behavior of the geodesics in a pseudoconvex smooth finite type domain in $\C^2$. First of all, we need to recall the following result about visibility (in this section the geodesics are with respect to
the Kobayashi distance).

\begin{theorem}\cite{BZ}\label{BhaZim} Let $\Omega\subset\C^2$ be bounded smooth pseudoconvex domain of finite type, fix $A\geq1, B\geq0$. If $\xi,\eta\in\partial\Omega$ and $V_\xi,V_\eta$ are neighborhood of $\xi,\eta\in \overline{\Omega}$ so that $\overline{V_\xi}\cap\overline{V_\eta}=\emptyset$ then there exists a compact set $K\subset\Omega$ with the following property. if $\sigma:[0,T]\longrightarrow\Omega$ is an $(A,B)$ quasi-geodesic with respect to the Kobayashi metric, with $\sigma(0)\in V_\xi$ and $\sigma(T)\in V_\eta$ then $\sigma\cap K\neq\emptyset$.
\end{theorem}

\begin{corollary}\label{gbgen}
Let $\Omega\subset\C^2$ be bounded smooth pseudoconvex domain of finite type, let $p_n, q_n\in\Omega$ be two sequence with $p_n\rightarrow\xi^+\in\partial\Omega$ and $q_n\rightarrow\xi^-\in\partial\Omega$, and 
$$\lim_{n,m\rightarrow\infty}(p_n|q_m)_o=\infty,$$
then $\xi^+=\xi^-$ (the Gromov product is with respect to $d_\Omega^K$).
\end{corollary}
\proof
Suppose by contradiction that $\xi^+\neq\xi^-$ and consider $\sigma_n:[a_n,b_n]\longrightarrow\Omega$ be a sequence of geodesics such that $\sigma_n(a_n)=p_n$ and $\sigma_n(b_n)=q_n$. Now from Theorem \ref{BhaZim}, there exists a compact set $K\subset\Omega$ and $T_n\in[a_n,b_n]$ such that $\sigma_n(T_n)\in K$. 
Finally
$$(p_n|q_n)_o=(\sigma_n(a_n)|\sigma_n(b_n))_o\leq d_\Omega^K(\sigma_n(T_n),o)\leq\max\{d_\Omega^K(x,o):x\in K\}<\infty$$
which is a contradiction.  
\endproof

\begin{proposition}\label{behgeogen}
Let $\Omega\subset\C^2$ be bounded smooth pseudoconvex domain of finite type and $\sigma:\R\longrightarrow\Omega$ be a geodesic line. Then both 
$$\lim_{t\rightarrow\infty}\sigma(t),\ \ \lim_{t\rightarrow-\infty}\sigma(t)$$
exist in $\partial{\Omega}$ and are different.
\end{proposition}
\proof
The existence of the limit is the same argument of Proposition \ref{behgeomode}.

For the other part, if by contradiction there exists a geodesic ray $\sigma:\R\longrightarrow\Omega$ such that $\lim_{t\rightarrow\infty}\sigma(t)= \lim_{t\rightarrow-\infty}\sigma(t)=:\xi$, after an affine transformation we can suppose that $\xi=(0,0)$ and $r$ is of the form (\ref{tayexp}). Let $R>0$ such that in $\Omega\cap B_R(0)$ the Catlin metric is well defined and it is equivalent to Kobayashi. Now define 
$$\tau^+:=\inf\{t\in[0,\infty): \sigma([t,\infty))\subset B_{R\slash2}(0)\}, $$
$$\tau^-:=\sup\{t\in(-\infty,0]: \sigma((-\infty,t])\subset B_{R\slash2}(0)\}.$$
Notice that  $\sigma^+:=\sigma|_{[\tau^+,\infty)}$, $\sigma^-:=\sigma|_{(-\infty,\tau^-]}$ are $(A,0)$ quasi-geodesics with respect to the Catlin metric.
Let $\eta_n:=(0,-n^{-1})$ and consider $\{\psi_n\}_{n\in\N}$ and $P:\C\longrightarrow\R$ be as in Theorem \ref{scalingth}. Denote $\tilde{\sigma}_n^+:=\psi_n\circ\sigma^+$ and $\tilde{\sigma}_n^-:=\psi_n\circ\sigma^-$, and notice that $\lim_{t\rightarrow\infty}\tilde{\sigma}^+_n(t)=\lim_{t\rightarrow-\infty}\tilde{\sigma}^-_n(t)=(0,0)$ and $\lim_{n\rightarrow\infty}\tilde{\sigma}^+_n(\tau^+)=\lim_{n\rightarrow\infty}\tilde{\sigma}^-_n(\tau^-)=\infty$, then using Proposition \ref{visibscal} there exists $T_n^+>0$ and $T_n^-<0$ such that $t\rightarrow \tilde{\sigma}_n^+(t+T^+_n)$ converges to a $(A,0)$ quasi-geodesic line $\hat{\sigma}^+:\R\longrightarrow\Omega_P$ and $t\rightarrow \tilde{\sigma}_n^-(t+T^-_n)$ converges to a $(A,0)$ quasi-geodesic line $\hat{\sigma}^-:\R\longrightarrow\Omega_P$. Finally,  $T_n^+\rightarrow\infty$ and $T_n^-\rightarrow-\infty$, then
\begin{align*}
d^C_{\Omega_P}(\hat{\sigma}^+(0),\hat{\sigma}^-(0))&=\lim_{n\rightarrow\infty}d^C_{r_n}(\tilde{\sigma}_n^+(T_n^+),\tilde{\sigma}_n^-(T_n^-))\\&=\lim_{n\rightarrow\infty}d^C_{r}(\sigma(T_n^+),\sigma(T_n^-))\\&\geq\lim_{n\rightarrow\infty} A^{-1}d^K_\Omega((\sigma(T_n^+),\sigma(T_n^-))\\&=\lim_{n\rightarrow\infty} A^{-1}|T_n^+-T_n^-|=\infty
\end{align*}
which is a contradiction. 
\endproof

Now we can prove the main Theorem

\begin{proof}[Proof of Theorem~\ref{mainth}]
\textbf{Part 1: $(\Omega,d^K_\Omega)$ is Gromov hyperbolic}

By contradiction, suppose that $(\Omega,d^K_{\Omega})$ is not Gromov hyperbolic, then there exist three sequences of points $\{x_n\}_{n\in\N},\{y_n\}_{n\in\N},\{z_n\}_{n\in\N}$, correspondent geodesic segment $\{\sigma_{x_n,y_n}\}_{n\in\N}$, $\{\sigma_{y_n,z_n}\}_{n\in\N}$, $\{\sigma_{x_n,z_n}\}_{n\in\N}$ and a point $u_n\in\sigma_{x_n,y_n}$ such that 
$$\lim_{n\rightarrow\infty}d^K_{\Omega}(u_n,\sigma_{x_n,z_n}\cup\sigma_{y_n,z_n})=\infty.$$
Up to passing to a subsequences, we can suppose that  $u_n\rightarrow\xi\in\overline{\Omega}$. We have two cases

\textbf{Case 1:} $\xi\in\Omega$.

First of all we can suppose that $x_n, y_n,z_n$ converge respectively to $x_\infty,y_\infty,z_\infty\in\overline{\Omega}$, and since
$$\lim_{n\rightarrow\infty} d^K_{\Omega}(u_n,\{x_n,y_n,z_n\})\geq\lim_{n\rightarrow\infty} d^K_{\Omega}(u_n,\sigma_{x_n,z_n}\cup\sigma_{y_n,z_n})=\infty$$
we have $x_\infty,y_\infty,z_\infty\in\partial\Omega$.
Then after a reparametrisation, we can suppose that $\sigma_{x_n,y_n}(0)=u_n$, and using Ascoli-Arzelà theorem $\sigma_{x_n,y_n}$ converges uniformly on compact sets to a geodesic line $\sigma:\R\longrightarrow\Omega$. Now by Proposition \ref{behgeogen} we have $x_\infty\neq y_\infty$, then one of the two is different from $z_\infty$, for example $x_\infty\neq z_\infty$. Now using Theorem \ref{BhaZim}, there exist a compact $K\subset\Omega$ and $T_n\in\R$ such that  $\sigma_{x_n,z_n}(T_n)\in K$. Finally
\begin{align*}
\max\{d^K_\Omega(\xi,x):x\in K\}&\geq\lim_{n\rightarrow\infty}d^K_\Omega(\sigma(0),\sigma_{x_n,z_n}(T_n))\\&=\lim_{n\rightarrow\infty}d^K_\Omega(\sigma_{x_n,y_n}(0),\sigma_{x_n,z_n}(T_n))\\&\geq \lim_{n\rightarrow\infty}d^K_\Omega(u_n,\sigma_{x_n,z_n}\cup\sigma_{y_n,z_n})=\infty
\end{align*}
and this is a contradiction.

\textbf{Case 2:} $\xi\in\partial\Omega$.

After an affine transformation, we can suppose that $\xi=(0,0)$ and $r$ is of the form (\ref{tayexp}). Using Theorem \ref{CatlinTh} there exists a $A\geq1$ and $R>0$, such that
$$A^{-1}M_r(p,v)\leq K_\Omega(p,v)\leq AM_r(p,v), \ \forall p\in B_R(0)\cap\Omega, v\in\C^2.$$
Let $\{\psi_n\}_{n\in\N}$ and $P$ be as in Theorem \ref{scalingth}, then
up to passing to subsequences $\psi_n(x_n), \psi_n(y_n)$ and $\psi_n(z_n)$ converge respectively to $x_\infty, y_\infty$ and $z_\infty$ in $\overline{\Omega_P}\cup\{\infty\}$. Now
we have to divide the problem in different cases. 

If both $x_\infty$ and $y_\infty$ are in $\partial\Omega_P$, then eventually $\sigma_{x_n,y_n}\subseteq B_R(0)$ and consequently $\psi_n\circ \sigma_{x_n,y_n}$ is an $(A,0)$ quasi-geodesic with respect to the Catlin metric that converges to an $(A,0)$ quasi-geodesic line $\sigma:\R\longrightarrow\Omega_P$ of ends point $x_\infty$ and $y_\infty$ (that are different, using Lemma \ref{behqgeomode}). Now after a labelling we can suppose that $x_\infty\neq z_\infty$, and we can define 
$$t_n=\sup\{t\in[a_n,b_n]: \sigma_{x_n,z_n}([a_n,t])\subset B_{R\slash2}(0)\},$$
and set $z'_n:=\sigma_{x_n,z_n}(t_n)$. Notice that 
$\lim_{n\rightarrow\infty}\psi_n(z'_n)=\lim_{n\rightarrow\infty}\psi_n(z_n)=z_\infty$ and $\sigma_{x_n,z'_n}:=\sigma_{x_n,z_n}|_{[a_n,t_n]}$ is in $B_R(0)$ and so it is a $(A,0)$ quasi-geodesic with respect to $M_r$.
Now by Proposition \ref{visibscal}, $\psi_n\circ\sigma_{x_n,z'_n}$ converges locally uniformly to the geodesic $\hat{\sigma}:\R\longrightarrow\Omega_P$. Thus,
\begin{align*}
d^C_{\Omega_P}((0,-1),\hat{\sigma}(0))&=\lim_{n\rightarrow\infty}d^C_{r}(u_n,\sigma_{x_n,z'_n}(0))\\&\geq A^{-1}\lim_{n\rightarrow\infty}d^K_{\Omega}(u_n,\sigma_{x_n,z'_n}(0))\\&\geq A^{-1}\lim_{n\rightarrow\infty}d^K_{\Omega}(u_n,\sigma_{x_n,z_n}\cup\sigma_{y_n,z_n})=\infty
\end{align*}

In the other case, we can always suppose that $x_\infty\in\partial\Omega_P$ and $y_\infty=\infty$. If $z_\infty=\infty$ we can find a $z_n''\in\sigma_{x_n,z_n}$ such that $\psi_n\circ\sigma_{x_n,z''_n}$ converges to a $(A,0)$ quasi-geodesic line, otherwise if $z_\infty\in\partial\Omega_P$ there exists $y'_n\in\sigma_{y_n,z_n}$ such that $\psi_n\circ\sigma_{y'_n,z_n}$ converges to a $(A,0)$ quasi-geodesic line, and using the previous argument we have a contradiction.

In all the cases we have a contradiction, then $(\Omega, d^K_{\Omega})$ is Gromov hyperbolic.

\textbf{Part 2: The identity map $\Omega\longrightarrow\Omega$
extends to a homeomorphism $\overline{\Omega}^G\longrightarrow\overline{\Omega}$.}

It is essentially identical to the proof of Proposition 1.2 in \cite{Zim1}, but for the reader's convenience we provide the complete argument.

We define the map
$$\Phi: \overline{\Omega}^G\longrightarrow\overline{\Omega}$$
as the identity in $\Omega$ and  $\Phi([\sigma])=\lim_{t\rightarrow+\infty}\sigma(t)$ if $[\sigma]\in\partial_G\Omega$.

First of all we prove that $\Phi$ is well defined. Indeed, given a geodesic ray $\sigma:[0,\infty)\longrightarrow\Omega$, $\lim_{t\rightarrow+\infty}\sigma(t)$ exists in $\C^2$ by Proposition \ref{behgeogen}. Let $\sigma_1$ and $\sigma_2$ be two geodesic rays with $\sigma_1(0)=\sigma_2(0)=o\in\Omega$ and 
$$\sup_{t\geq0}d^K_\Omega(\sigma_1(t),\sigma_2(t))<\infty$$
then
$$\lim_{t\rightarrow+\infty}(\sigma_1(t)|\sigma_2(t))_o=\infty$$
which implies by Corollary \ref{gbgen} that
$$\lim_{t\rightarrow+\infty}\sigma_1(t)=\lim_{t\rightarrow+\infty}\sigma_2(t),$$
thus $\Phi$ is well defined.
Now since $\overline\Omega$ is compact, it is sufficient to prove that $\Phi$ is continuous, injective and surjective. 

For the continuity, we prove that if $\xi_n\rightarrow\xi\in\partial\Omega$ then $\Phi(\xi_n)\rightarrow\Phi(\xi)$. Fix $o\in\Omega$ and consider for each $n\in\N$ $T_n\in[0,+\infty]$ and $\sigma_n:[0,T_n)\rightarrow\Omega$ such that $\sigma_n(0)=o$ and $\lim_{t\rightarrow T_n}\sigma_n(t)=\xi_n$. Since $T_n$ could be $\infty$, let $T'_n\in(0,T_n)$ such that 
$$\lim_{n\rightarrow\infty}\sigma_n(T'_n)=\lim_{n\rightarrow\infty}\Phi(\xi_n)=\xi.$$ Now by Ascoli-Arzelà theorem, $\sigma_n$ (up to subsequences) converges locally uniformly to a geodesic ray $\sigma:[0,+\infty)\longrightarrow\Omega$, and with a similar argument used in Corollary \ref{limgeo} 
$$\lim_{t\rightarrow+\infty}\sigma(t)=\lim_{n\rightarrow\infty}\Phi(\xi_n)=\xi$$
so $\Phi$ is continuous. 

For the injectivity, let $\sigma_1,\sigma_2$ be two geodesic rays such that $\lim_{t\rightarrow+\infty}\sigma_1(t)=\lim_{t\rightarrow+\infty}\sigma_2(t)=\xi\in\partial\Omega$ but $[\sigma_1]\neq[\sigma_2]$ in $\overline{\Omega}^G$. Now, recalling that every two distinct points in $\partial^G\Omega$ can be joint by a geodesic line, there exists a geodesic line $\sigma:\R\longrightarrow\Omega$ such that 
$$\sup_{t\geq0}d_\Omega^K(\sigma(-t),\sigma_1(t))<K\ \ \mbox{and}\ \ \sup_{t\geq0}d_\Omega^K(\sigma(t),\sigma_2(t))<K$$
and consequentially  
$$\lim_{t\rightarrow+\infty}(\sigma(-t)|\sigma_1(t))_o=\lim_{t\rightarrow+\infty}(\sigma(t)|\sigma_2(t))_o=\infty $$
that implies using Corollary \ref{gbgen} that 
$$\lim_{t\rightarrow-\infty}\sigma(t)=\lim_{t\rightarrow\infty}\sigma_1(t)=\lim_{t\rightarrow\infty}\sigma_2(t)=\lim_{t\rightarrow\infty}\sigma(t)$$
but this contradicts Proposition \ref{behgeogen}.

Finally for the surjectivity, let $\xi\in\partial\Omega$  and $x_n\in\Omega\rightarrow\xi$. Fix $o\in\Omega$ and consider $\sigma_n:[0,T_n]\longrightarrow\Omega$ such that $\sigma_n(0)=o$ and $\sigma_n(T_n)=x_n$. By Ascoli-Arzelà theorem, $\sigma_n$ (up to subsequences) converges locally uniformly to a geodesic ray $\sigma:[0,+\infty)\longrightarrow\Omega$, and with the same argument used in Corollary \ref{limgeo} 
$$\lim_{t\rightarrow+\infty}\sigma(t)=\lim_{n\rightarrow\infty}\sigma_n(T_n)=\xi,$$
so $\Phi([\sigma])=\xi$, then $\Phi$ is surjective.
\end{proof}

\section{Applications and examples}

In this last section we provide some applications.

We start with the problem of the extension to the boundary of the biholomorphisms: using Theorem \ref{mainth} and \ref{extth} we obtain 

\begin{corollary}
Let $\Omega_1,\Omega_2\subset\C^2$ be a bounded smooth finite type domain and let $f:\Omega_1\longrightarrow\Omega_2$ be a biholomorphic map, then $f$ extends to a unique homeomorphism between $\overline{\Omega}_1$ and $\overline{\Omega}_2$.
\end{corollary}

The smooth extension of biholomorphisms between pseudoconvex finite type domains in $\C^2$ was already proved by Bell in \cite{Bell} using the results of Kohn \cite{Kohn} about the global regularity estimates of $\bar{\partial}$-Neumann operator.

More interesting is the study the biholomorphism between finite type domains and convex domains (we include also unbounded convex domains). 

We recall that a convex domain $\Omega$ is $\C$-proper if it does not contain any complex affine lines (in \cite{BS} it is proved that this is equivalent to the Kobayashi completeness of $\Omega$). For convex domains we have a natural definition of compactification, the \textit{end compactification}, that we denote with $\overline{\Omega}^*$. The definition is the following (see, for example, \cite{BG1, BGZ} for  details): if $\Omega$ is bounded, then $\overline{\Omega}^*=\overline{\Omega}$. If $\Omega$ is unbounded, then there exists $R_0>0$ such that for each $R>R_0$ the set $\Omega\cap B(0,R)^c$ has one or two connected components and the number of components does not depend on $R$ (we call $\Gamma$ the set of connected components of $\Omega\cap B(0,R_0)^c$).  We define $\overline{\Omega}^*:=\overline{\Omega}\cup\Gamma$. For the topology, we say $\xi_n\in\overline{\Omega}^*\rightarrow\xi\in\overline{\Omega}^*$ is the usual Euclidean topology if $\xi\in\overline{\Omega}$, instead if $\xi\in\Gamma$ we have that $\xi_n\in \overline{\xi}\cup\{\xi\}$ eventually. 

The interesting fact is that the end compactification (that can be defined in every convex set) is the same of Gromov compactification if $(\Omega,d^K_\Omega)$ is Gromov hyperbolic. 

\begin{theorem}\cite{BGZ}\label{convexgromov} Let $\Omega\subset\C^d$ be a $\C$-proper convex domain such that $(\Omega,d^K_\Omega)$ is Gromov hyperbolic. Then $\overline{\Omega}^G$ and $\overline{\Omega}^*$ are homeomorphic.
\end{theorem}

Finally we can enunciate the following 

\begin{theorem}\label{Thm:extension}
Let $\Omega_1\subset\C^2$ be a bounded smooth finite type domain and $\Omega_2\subset\C^2$ be a convex domain. Let $f:\Omega_1\longrightarrow\Omega_2$ be a biholomorphic map, then $f$ extends to a homeomorphism between $\overline{\Omega}_1$ and $\overline{\Omega}_2^*$.
\end{theorem}
\proof
Since $\Omega_1$ is Kobayashi hyperbolic and Gromov hyperbolic and $f$ is an isometry for Kobayashi, then also $\Omega_2$ is Kobayashi hyperbolic (and so it is $\C$-proper) and Gromov hyperbolic. Now from Theorem \ref{convexgromov} $\overline{\Omega}^G$ and $\overline{\Omega}^*$ are homeomorphic. We conclude using Theorem \ref{extth}.
\endproof

An interesting application of Gromov hyperbolicity is in complex dynamics. Karlsson \cite{Kar} proved that in  Gromov hyperbolic spaces there exists a general Denjoy-Wolff Theorem for 1-Lipschitz maps. Another topic is the study of the dynamics of commuting holomorphic self-mappings. In \cite{Behan} Behan proved that two commuting holomorphic self-maps of the unit disc with no inner fixed points either have the same Denjoy-Wolff point or they are both hyperbolic automorphisms. Bracci \cite{Br1, Br2} extended such a result to commuting holomorphic self-maps of the ball and of convex domains. 
In \cite{BGZ} the authors proved a general version in Gromov hyperbolic spaces and apply it to get Behan's result to strongly pseudoconvex domains and Gromov hyperbolic convex domains. Now we give a similar result for bounded pseudoconvex finite type domains in $\C^2$. We recall that a complex geodesic is a holomorphic function $\phi:\D\twoheadrightarrow\Delta\subseteq\Omega$ such that
$d^K_\Omega(\phi(z),\phi(w))=\omega(z,w)$ for each $z,w\in\D$ where $\omega$ is the Poincarè distance on $\D$ (that means that $\phi$ is an isometric embedding). A holomorphic retraction is a holomorphic function $\rho:\Omega\twoheadrightarrow M\subseteq\Omega$ with $\rho^2=\rho$ (or equivalently, $\rho|_M=\id_M$). It is possible to prove that $M$ is always a complex sub-manifold of $\C^d$.

\begin{remark}
Let $\Omega\subset\C^d$ be a domain and $\rho:\Omega\twoheadrightarrow M\subseteq\Omega$ a retraction. Then using Proposition \ref{Kobdecr}, for each $z,w\in M$
$$d^K_M(z,w)\geq d^K_\Omega(z,w)\geq d^K_M(\rho(z),\rho(w))=d^K_M(z,w)$$
that means that $\rho$ is an isometry.
\end{remark}

\begin{theorem}\label{Thm:Behan}
Let $\Omega\subset\C^2$ be a bounded smooth pseudoconvex finite type domain, and let $f,g$ be commuting holomorphic maps from $\Omega$ to $\Omega$. Suppose that there exist $\xi\neq\eta$ in $\partial\Omega$ such that
$$f^n\rightarrow\xi, \ \ \ g^n\rightarrow\eta.$$ Then there is a complex geodesic $\Delta$ of $\Omega$, which is a retract of $\Omega$, such that $\xi,\eta\in\partial\Delta$ and $f|_\Delta,g|_\Delta$ are hyperbolic automorphism of $\Delta$.
\end{theorem}
\proof
Using Theorem 1.9 of \cite{BGZ} and Theorem \ref{mainth} we have that there exists a retraction $\rho:\Omega\twoheadrightarrow M\subseteq\Omega$ such that $f|_M,g|_M\in$Aut$(M)$. First of all, notice that $M$ is not a point, since $f$ and $g$ have no fixed points in $\Omega$. We analyse two different cases, according to the dimension of $M$
\begin{itemize}
\item $dimM=2$: then $M=\Omega$ and $f,g\in$Aut$(\Omega)$. By Bedford-Pinchuck Theorem \cite{BP2} $\Omega$ is biholomorphic to a Thullen domain $T:=\{(z,w)\in\C^2:\re[w]+|z|^{2p}<0 \}$. The automorphisms group of $T$ is well-know (see for instance \cite{automodel}) then with a simple computation it is possible to find an automorphism that conjugate $f$ and $g$ respectively in
$$\tilde{f}(z,w)=(\lambda z,\lambda^{2p} w) \ \ \mbox{and} \ \ \tilde{g}(z,w)=(\mu^{-1} z,\mu^{-2p} w)$$
with $\lambda,\mu\in(0,1)$. Now the complex geodesic sought is $\phi:\D\longrightarrow\Delta\subset T$ given by $$\phi(\zeta)=\left(0,\frac{\zeta+1}{\zeta-1}\right),$$ which is the retract of the (linear) retraction $\pi:T\twoheadrightarrow\Delta\subset T$, $\pi(z,w)=(0,w)$.
\item $dimM=1$: we want to prove that $M$ is simply connected, in this way the statement follows from the Uniformization Theorem. Since $\Omega$ is smooth, we can find a neighborhood $U$ of $\xi$ such that $U\cap\Omega$ is simply connected. Suppose that $M$ is not simply connected, then there exists a closed path $\gamma$ such that is not homotopic to 0 in $M$. Since $\gamma$ is compact and $f^n\rightarrow\xi$, there exist $n_0$ such that $\tilde{\gamma}:=f^{n_0}\circ\gamma$ is in $U\cap M$ and it is not homotopic to 0 in $M$ (because $f$ is an automorphism). Now $\rho$ is a retract, then $i^*:\pi_1(M)\longrightarrow\pi_1(\Omega)$ is injective and so $\tilde{\gamma}$ is not homotopic to 0 in $\Omega$, that implies that it is not homotopic to 0 also in $U\cap\Omega$ but this is a contradiction because $U\cap\Omega$ is simply connected.
\end{itemize}
\endproof

We conclude with an example of a theorem that cannot be derived from the results known so far.

\begin{example}
Let
$$D:=\left\{(z,w)\in\C^2: \re[w]+\frac{|w|^2}{5}+|zw|^2+|z|^8+\frac{15}{7}|z|^2\re[z^6]+10|z|^{10}<0\right\}.$$
It is a bounded pseudoconvex finite type domain in $\C^2$, so by Theorem \ref{mainth} $(D,d_D^K)$ is Gromov hyperbolic. Moreover $D$ is not convex and Calamai in \cite{Cal} proved that there does not exist a neighbourhood of the origin $V$ such that $V\cap D$ is biholomorphic to a convex domain, so $D$ is not locally convexifiable in the sense of \cite{Zim1}.
\end{example}

\end{document}